\documentclass{article}
\usepackage{amsmath,amssymb,amsthm}
\usepackage[cp1250]{inputenc}
\usepackage{enumerate}
\usepackage{a4wide}

\usepackage{color} 

\numberwithin{equation}{section}

\def\XXint#1#2#3{{\setbox0=\hbox{$#1{#2#3}{\int}$ }
\vcenter{\hbox{$#2#3$ }}\kern-.6\wd0}}

\date{}

\newcommand{\de}{\partial}
\newcommand{\pat}{\partial_t}

\newcommand{\paxx}{\partial_{x_2}}

\newcommand{\dxdt}{\ {\rm d}x{\rm d}t}

\newcommand{\abs}[1]{\ensuremath{\left| #1 \right|}}
\newcommand{\ep}{\varepsilon}
\newcommand{\epd}{\ep_\Delta}
\newcommand{\dd}{\delta_\Delta}
\newcommand{\id}{\mathbb{I}}
\newcommand{\Sz}{\mathcal{S}^{2\times 2}_0}
\newcommand{\R}{\mathbb{R}} 
\newcommand{\Div}{{\rm div}_x}

\renewcommand{\d}{{\rm d}}

\newcommand{\dr}{\,{\rm d}r}
\newcommand{\dt}{\,{\rm d}t}
\newcommand{\dtau}{\,{\rm d}\tau}

\newcommand{\Sym}{\mathcal{S}}
\newcommand{\N}{\mathbb{N}}

\newtheorem{theorem}{Theorem}[section]
\newtheorem{lemma}[theorem]{Lemma}
\newtheorem{definition}[theorem]{Definition}
\newtheorem{propo}[theorem]{Proposition}
\newtheorem{remark}[theorem]{Remark}

\newtheorem{coro}[theorem]{Corollary}

\title{Non--uniqueness of admissible weak solutions to the compressible Euler equations with smooth initial data}
\author{Elisabetta Chiodaroli$^2$\footnote{elisabetta.chiodaroli@unipi.it} \and Ond\v rej Kreml$^1$\footnote{kreml@math.cas.cz} \and
V\'aclav M\'acha$^1$\footnote{macha@math.cas.cz} 
\and 
Sebastian Schwarzacher$^1$\footnote{schwarz@math.cas.cz}
}

\begin{document}


\maketitle

\bigskip

\centerline{$^1$Institute of Mathematics, Czech Academy of Sciences}

\centerline{\v Zitn\' a 25, 115 67 Praha 1, Czech Republic}
\medskip

\centerline{$^2$Dipartimento di Matematica, Universit\`a di Pisa}
\centerline{Via F. Buonarroti 1/c, 56127 Pisa, Italy}

\bigskip

\textbf{Abstract:} We consider the isentropic Euler equations of gas dynamics in the whole two-dimensional space and we prove the existence of a $C^\infty$ initial datum which admits infinitely many bounded admissible weak solutions. Taking advantage of the relation between smooth solutions to the Euler system and to the Burgers equation we construct a smooth compression wave which collapses into a perturbed Riemann state at some time instant $T > 0$. In order to continue the solution after the formation of the discontinuity, we adjust and apply the theory developed by De Lellis and Sz\'ekelyhidi in \cite{DLSz1}--\cite{DLSz2} and we construct infinitely many solutions. We introduce the notion of an admissible generalized fan subsolution to be able to handle data which are not piecewise constant and we reduce the argument to finding a single generalized subsolution.


\section{Introduction} 

In this paper, we study  the Cauchy problem for the isentropic Euler equations of gas dynamics in two space dimensions, i.e. we look for a density $\rho(t,x)$ and a velocity $v(t,x)$ satisfying 
\begin{align}
\pat \rho + \Div (\rho v) &= 0 \label{eq:EulerCE}\\
\pat (\rho v) + \Div (\rho v \otimes v) + \nabla_x p(\rho) &= 0 \label{eq:EulerME}
\end{align}
in $(0,T) \times \R^2$ together with the initial conditions
\begin{equation}
(\rho,v)(0,\cdot) = (\rho^0, v^0)(\cdot) \qquad \text{ in } \R^2. \label{eq:EulerIC}
\end{equation}
The system \eqref{eq:EulerCE}-\eqref{eq:EulerME} is a paradigmatic example of a hyperbolic system of conservation laws. The mathematical theory of such systems of partial differential equations is still far from being completely understood in more than one space dimension. It is however well known and observed already in the case of the most simple example of a scalar law, the Burgers equation, that singularities appear even in case of smooth initial data and that weak solutions with no additional properties are not unique. The notion of admissibility in the form of entropy inequalities was used for one-dimensional systems and multi-dimensional scalar equations to restore uniqueness of weak solutions at least in proper classes of solutions; for more details about this nowadays classical theory we refer for example to the monograph \cite{Dafermos16}.

This theory in particular motivates us to look for weak solutions (i.e. solutions in the sense of distributions) which satisfy in addition the entropy inequality. In the case of the system \eqref{eq:EulerCE}-\eqref{eq:EulerME} this is rather a form of energy inequality, since the only existing mathematical entropy is the physical total energy of the system. Therefore, we call a weak solution \textit{admissible}, if it satisfies the energy inequality
\begin{equation} \label{eq:EulerEnergyIneq}
\pat \left(\rho e(\rho) + \rho\frac{\abs{v}^2}{2}\right) + \Div \left(\left(\rho e(\rho) + p(\rho) + \rho\frac{\abs{v}^2}{2}\right) v \right) \leq 0,
\end{equation}
in the sense of distributions on $[0,T) \times \R^2$.
In \eqref{eq:EulerEnergyIneq}, $e(\rho)$ denotes the internal energy that is related to the pressure through the relation
\begin{equation}
p(r) = r^2 e'(r).
\end{equation}

In their groundbreaking work \cite{DLSz1}, \cite{DLSz2} on the existence of infinitely many solutions to the incompressible Euler equations in multiple space dimensions, De Lellis and Sz\'ekelyhidi also noted that this theory can be applied to the compressible isentropic system as well and proved in \cite{DLSz2} the existence of initial data $(\rho^0,v^0) \in L^\infty(\R^2)$ for which there exists infinitely many admissible weak solutions to \eqref{eq:EulerCE}-\eqref{eq:EulerIC}. This result was further improved by the first author in \cite{Ch} and by Feireisl in \cite{Fe}, proving the existence of initial data $\rho^0 \in C^1(\R^2), v^0 \in L^\infty(\R^2)$ for which there exists infinitely many admissible weak solutions, locally in time and globally in time respectively. We call such data allowing for infinitely many admissible weak solutions \textit{wild data}.

Quite surprisingly in \cite{ChiDelKre15}, the first two authors together with De Lellis proved that even in the class of Lipschitz initial data there exists an example of wild data. The proof in \cite{ChiDelKre15} relied substantially on the convex integration method developed by De Lellis and Sz\'ekelyhidi \cite{DLSz1}, \cite{DLSz2} for the incompressible Euler equations combined with a clever analysis of the Riemann problem for \eqref{eq:EulerCE}-\eqref{eq:EulerME}. A building block in the construction carried out in \cite{ChiDelKre15} are piecewise constant subsolutions to \eqref{eq:EulerCE}-\eqref{eq:EulerME} which were also inspired by the work \cite{sz} of Sz\'ekelyhidi on irregular solutions of the incompressible Euler equations with vortex-sheet data. The wild initial data in \cite{ChiDelKre15} are chosen so to generate a compression wave collapsing to a shock giving then rise to a particular Riemann datum at a certain time instant $t = T$ starting from which infinitely many admissible solutions forward in time can be constructed. Such a solution $(\rho,v)$ is indeed Lipschitz on the time interval $(0,T)$ but not smoother, since the first derivatives are not continuous on the boundaries of the wave.

A natural question therefore arises, whether a similar property can be achieved also for even more regular initial data. In this work, we provide the positive answer. Even so the answer itself is hardly surprising the technical machinery that was necessary to be built in order to produce these counterexamples is rather involving; indeed, we provide technical tools in the proof which may well be of further use.

Our main theorem reads as follows.
\begin{theorem}\label{t:main}
Let $p(\rho) = \rho^2$ and let $T > 0$. There are initial data $(\rho^0,v^0) \in C^\infty(\R^2)$ and $\delta_0 > 0$, such that there exist infinitely many bounded admissible weak solutions to the Euler system \eqref{eq:EulerCE}-\eqref{eq:EulerIC} on the time interval $(0,T+\delta_0)$. These solutions all coincide with the unique $C^\infty$ solution to the Euler system \eqref{eq:EulerCE}-\eqref{eq:EulerIC} on the time interval $(0,T)$.
\end{theorem}

Even if $C^\infty$ data are shown to allow for infinitely many bounded admissible solutions, such solutions still display a high level of irregularity, i.e. they are only bounded and highly oscillatory. An open question for the compressible Euler equations concerns the level of regularity of solutions displaying such a non--standard behavior as in Theorem \ref{t:main}. We note that the related question concerning regularity of solutions needed for the energy conservation was recently studied in \cite{FGGW}.

This issue has been now extensively tackled for the incompressible Euler equations leading to the proof of the famous Onsager conjecture. For details see
\cite{Eyink}, \cite{CET} for one side of the conjecture and \cite{DLSz3}, \cite{Bu}, \cite{BDIS}, \cite{BDS}, \cite{BDSV}, \cite{DS}, \cite{Is} and \cite{Is2} for the other side. 

The paper is structured as follows. In Section \ref{s:cw} we introduce the notion of smooth compression wave, which is one of the building blocks in our proof. In Section \ref{s:ss} we define subsolutions in such a way that the proof of Theorem \ref{t:main} is reduced to finding a single subsolution so to apply the theory of De Lellis and Sz\'ekelyhidi. In our case we generalize the definition of subsolution given in \cite{ChiDelKre15} since the assumption that the subsolutions consist of piecewise constant functions is too restrictive for our construction. The main ingredient to the proof of the Proposition \ref{p:main} which states the relation between the existence of infinitely many admissible solutions and the existence of a single properly defined subsolution is Lemma \ref{l:CI} which is a straightforward generalization of \cite[Lemma 3.7]{ChiDelKre15}. The main technical part of the proof of Theorem \ref{t:main} is contained in Section \ref{s:main}.

\begin{remark}
In this paper, it is useful to work both with the velocity $v$ and the momentum $m = \rho v$. We emphasize here that we never run into troubles arising from possible presence of vacuum regions using either of these notions, since we always work with strictly positive densities. 
\end{remark}

\section{Smooth compression wave}\label{s:cw}

In this section we recall some standard facts about the one-dimensional rarefaction waves solving the Riemann problem for the Euler system and introduce the smooth compression wave. Since we want to keep this section related to the two-dimensional case, we denote $v = (v_1,v_2)$, $m = (m_1,m_2)$ and we work here with the system for unknowns $(\rho,v_2)$ in the space-time domain $(t, x_2) \in \R^+\times \R$, i.e.
\begin{align}
\pat \rho + \paxx (\rho v_2) &= 0 \label{eq:1DEE1}\\
\pat (\rho v_2) + \paxx (\rho v_2^2 + p(\rho)) &= 0\label{eq:1DEE2} \\
(\rho,v_2)(0,\cdot) &= (\rho^0, v_2^0).\label{eq:1DEEIC}
\end{align}
It is well known that the characteristic speeds of this system are
\begin{equation}\label{eq:CharSpeed}
\lambda_1 = v_2 - \sqrt{p'(\rho)}, \qquad \lambda_2 = v_2 + \sqrt{p'(\rho)}
\end{equation}
and the functions 
\begin{equation}\label{eq:RiemInv}
w_1 = v_2 + \int_0^\rho \frac{\sqrt{p'(r)}}{r}\dr, \qquad w_2 = v_2 - \int_0^\rho \frac{\sqrt{p'(r)}}{r}\dr
\end{equation}
are, respectively, $1$- and $2$-Riemann invariants. The classical theory, see for example \cite[Theorem 7.6.6]{Dafermos16}, yields that every $i$-Riemann invariant is constant along any $i$-rarefaction wave. It is then not difficult to observe that for example in the case of a $1$-rarefaction wave, under the condition $w_1 \equiv const$, the system \eqref{eq:1DEE1}-\eqref{eq:1DEEIC} reduces to a simple Burgers equation for $\lambda_1$
\begin{align}
\pat \lambda_1 + \lambda_1\paxx\lambda_1 &= 0 \label{eq:Bur} \\
\lambda_1(0,\cdot) &= \lambda_1^0. \label{eq:BurIC}
\end{align}
This known property has been used also in \cite{ChiGo17} to set up a first numerical simulation of non--standard solutions to isentropic Euler. 
As it is well known, the solution to the Burgers equation can be obtained by the method of characteristics as long as the characteristic curves do not intersect each other. The characteristic curves are defined as
\begin{equation}
x = \lambda_1^0(r)t + r,
\end{equation}
for $r \in \R$, and the solution $\lambda_1$ to the Burgers equation \eqref{eq:Bur} is constant along these curves, taking values $\lambda_1^0(r)$. The classical theory also yields, that for the Riemann problem
\begin{equation}
\lambda_1^0(x_2) = \left\{\begin{split}
&\lambda_- \qquad \text{ for } x_2 < 0 \\
&\lambda_+ \qquad \text{ for } x_2 > 0,
\end{split} \right.
\end{equation}
with $\lambda_- < \lambda_+$, the problem \eqref{eq:Bur}-\eqref{eq:BurIC} admits a Lipschitz solution in the form of a rarefaction wave
\begin{equation}
\lambda_1(t,x_2) = \left\{\begin{split}
&\lambda_- \qquad \text{ for } x_2 < \lambda_- t \\
& \frac{x_2}{t} \qquad \text{ for } \lambda_- t < x_2 <  \lambda_+ t \\
&\lambda_+ \qquad \text{ for } x_2 > \lambda_+ t.
\end{split} \right.
\end{equation}
Reversing time and space we obtain a solution to the Burgers equation called the compression wave. Namely for initial data 
\begin{equation}
\lambda_1^0(x_2) = \left\{\begin{split}
&\lambda_- \qquad \text{ for } x_2 < -\lambda_- T \\
& -\frac{x_2}{T} \qquad \text{ for } -\lambda_- T < x_2 <  -\lambda_+ T \\
&\lambda_+ \qquad \text{ for } x_2 > -\lambda_+ T
\end{split} \right.
\end{equation}
for some $T > 0$ and $\lambda_- > \lambda_+$, the solution to \eqref{eq:Bur}-\eqref{eq:BurIC} on the time interval $(0,T)$ is Lipschitz and has the form
\begin{equation}
\lambda_1(t,x_2) = \left\{\begin{split}
&\lambda_- \qquad \text{ for } x_2 < -\lambda_- (T-t) \\
& -\frac{x_2}{T-t} \qquad \text{ for } -\lambda_- (T-t) < x_2 <  -\lambda_+ (T-t) \\
&\lambda_+ \qquad \text{ for } x_2 > -\lambda_+ (T-t),
\end{split} \right.
\end{equation}
in time $t = T$ the solution has the form
\begin{equation}
\lambda_1(T,x_2) = \left\{\begin{split}
&\lambda_- \qquad \text{ for } x_2 < 0 \\
&\lambda_+ \qquad \text{ for } x_2 > 0,
\end{split} \right.
\end{equation}
and for $t > T$ the solution consists of a shock.

We introduce the smooth compression wave as follows. Consider initial data in the form 
\begin{equation}\label{eq:BurICsmoothCW}
\lambda_1^0(x_2) = \left\{\begin{split}
&\lambda_- \qquad \text{ for } x_2 < -\lambda_- T - \zeta_1 \\
&f_-^0(x_2) \qquad \text{ for } -\lambda_- T - \zeta_1 < x_2 < -\lambda_- T + \zeta_2 \\
& -\frac{x_2}{T} \qquad \text{ for } -\lambda_- T + \zeta_2 < x_2 <  -\lambda_+ T - \zeta_2 \\
&f_+^0(x_2) \qquad \text{ for } -\lambda_+ T - \zeta_2 < x_2 < -\lambda_+ T + \zeta_1 \\
&\lambda_+ \qquad \text{ for } x_2 > -\lambda_+ T + \zeta_1,
\end{split} \right.
\end{equation}
where $\zeta_1,\zeta_2 > 0$ are sufficiently small (we always assume $\zeta_{1,2} < 1$) and $f_\pm^0$ are smooth strictly monotone functions with strictly monotone first derivatives such that $\lambda_1^0 \in C^\infty(\R)$. The solution to the Burgers equation \eqref{eq:Bur}-\eqref{eq:BurIC} then has the following form on the time interval $(0,T)$
\begin{equation}\label{eq:BursmoothCW}
\lambda_1(t,x_2) = \left\{\begin{split}
&\lambda_- \qquad \text{ for } x_2 < -\lambda_- (T-t) - \zeta_1 \\
&f_-(t,x_2) \qquad \text{ for } -\lambda_- (T-t) - \zeta_1 < x_2 < -\lambda_- (T-t) + \zeta_2\frac{T-t}{T} \\
& -\frac{x_2}{T-t} \qquad \text{ for } -\lambda_- (T-t) + \zeta_2\frac{T-t}{T} < x_2 <  -\lambda_+ (T-t) - \zeta_2\frac{T-t}{T} \\
&f_+(t,x_2) \qquad \text{ for } -\lambda_+ (T-t) - \zeta_2\frac{T-t}{T} < x_2 < -\lambda_+ (T-t) + \zeta_1 \\
&\lambda_+ \qquad \text{ for } x_2 > -\lambda_+ (T-t) + \zeta_1
\end{split} \right.
\end{equation}
and in particular at time $t=T$ a discontinuity appears
\begin{equation}\label{eq:BursmoothCWT}
\lambda_1(T,x_2) = \left\{\begin{split}
&\lambda_- \qquad \text{ for } x_2 < - \zeta_1 \\
&f_-(T,x_2) \qquad \text{ for } - \zeta_1 < x_2 < 0 \\
&f_+(T,x_2) \qquad \text{ for } 0 < x_2 <  \zeta_1 \\
&\lambda_+ \qquad \text{ for } x_2 >  \zeta_1.
\end{split} \right.
\end{equation}
Here again the functions $f_\pm(t,\cdot)$ are smooth and strictly monotone with strictly monotone first derivatives for all $t \in (0,T]$ and $\lambda_1 \in C^\infty((0,T)\times \R)$. It is also easy to observe that the left derivative of $f_-(T,\cdot)$ and the right derivative of $f_+(T,\cdot)$ at point $x_2 = 0$ are equal to $-\infty$.

However it is not immediately clear what the functions $f_\pm(T,x_2)$ need to satisfy in order to ensure that they emerge from $C^\infty$ initial data $\lambda_1^0$ as in \eqref{eq:BurICsmoothCW}. The question is obviously the behaviour of $f_\pm(T,x_2)$ near the origin, where the discontinuity appears. We provide an example of such functions in the following lemma.

\begin{lemma}\label{l:f0}
Let
\begin{equation}\label{eq:f0def}
    f_0(x_2) = 1- \frac{2}{\pi}\arctan(\log|\log(x_2)|).
\end{equation}
Let 
\begin{align}
 f_-(T,x_2) &= a_- f_0(-x_2) + b_-   \\
 f_+(T,x_2) &= -a_+f_0(x_2) + b_+
\end{align}
for $0 < |x_2| < \overline{\zeta}$ with $\overline{\zeta} < \frac{\zeta_1}{2}$ and with some $a_\pm > 0$ and $b_\pm = \lambda_\pm \pm \frac{\zeta_2}{T}$. For $|x_2| \geq \overline{\zeta}$ define $f_\pm(T,x_2)$ in such a way that $\lambda_1(T,\cdot)$ defined as in \eqref{eq:BursmoothCWT} is smooth and strictly monotone with strictly monotone first derivatives on intervals $x_2 \in (-\infty,0)$ and $x_2 \in (0,\infty)$.
Then there exists initial data $\lambda_1^0 \in C^\infty(\R)$ as in \eqref{eq:BurICsmoothCW} such that the solution $\lambda_1(t,x_2)$ to the Burgers equation \eqref{eq:Bur}-\eqref{eq:BurIC} satisfies $\lambda_1 \in C^\infty((0,T)\times \R)$ and $\lambda_1(T,x_2)$ has the form \eqref{eq:BursmoothCWT}.
\end{lemma}
\begin{proof}
It is enough to show that taking $f_0(x_2)$ as initial data for the Burgers equation
\begin{align}
    \pat f + f \paxx f &= 0 \label{eq:BurLemma21}\\
    f(0,x_2) &= f_0(x_2) \label{eq:BurLemma21IC}
\end{align}
on the right neighborhood of zero will yield that for any $t > 0$ the solution to the Burgers equation $f(t,x_2)$ will have the properties
\begin{align}
    \lim_{x_2 \rightarrow 0_+} \frac{\de f(t,x_2)}{\de x_2} &= \frac{1}{t} \label{eq:prvnider}\\
    \lim_{x_2 \rightarrow 0_+} \frac{\de^{(n)} f(t,x_2)}{\de x_2^n} &= 0 \label{eq:ntader}
\end{align}
for all $n \geq 2$ and therefore it can be smoothly connected to a linear function $\frac{x_2}{t}$.

Since $f_0(x_2)$ is an increasing function, the solution $f(t,x_2)$ to the Burgers equation in the right neighborhood of zero is obtained by the method of characteristics. The characteristic lines are 
\begin{equation}\label{eq:charlines}
    x_2 = f_0(r)t + r
\end{equation}
for $r \geq 0$. We introduce the inverse function to \eqref{eq:charlines} $r = g(t,x_2)$, i.e. it holds
\begin{equation}\label{eq:charinverse}
    x_2 = g(t,x_2) + t f_0(g(t,x_2)) 
\end{equation}
for $t > 0$ and $x_2$ in the right neighborhood of zero. In particular it is easy to observe that $g(t,0) = 0$ for any $t > 0$, $g(t,x_2)$ is smooth and $g(t,x_2) \neq 0$ for $x_2 > 0$. From \eqref{eq:charinverse} we express
\begin{equation}\label{eq:paxg}
    \paxx g(t,x_2) = \frac{1}{1+tf_0'(g(t,x_2))}.
\end{equation}
The solution to the Burgers equation is given by 
\begin{equation}
    f(t,x_2) = f_0(g(t,x_2))
\end{equation}
and therefore using \eqref{eq:paxg}
\begin{equation}\label{eq:paxftx}
    \paxx f(t,x_2) = \frac{f_0'(g(t,x_2))}{1+tf_0'(g(t,x_2))}.
\end{equation}
On the other hand 
\begin{equation}\label{eq:f0der}
    f_0'(y) = -\frac{2}{\pi}\frac{1}{1+(\log|\log(y)|)^2}\frac{1}{y\log(y)},
\end{equation}
therefore obviously $\lim_{y\rightarrow 0_+} f_0'(y) = \infty$ and \eqref{eq:paxftx} yields \eqref{eq:prvnider}.

We continue by calculating the second derivative of $f(t,x_2)$ and we obtain
\begin{equation}\label{eq:pax2ftx}
    \paxx^2 f(t,x_2) = \frac{f_0''(g(t,x_2))}{(1+tf_0'(g(t,x_2)))^3}.
\end{equation}
while an easy observation yields 
\begin{equation}\label{eq:f0der2}
    f_0''(y) = \frac{1}{y^2}R_2(\log(y),\log|\log(y)|)
\end{equation}
for some rational function $R_2$. Plugging in \eqref{eq:f0der2} and \eqref{eq:f0der} into \eqref{eq:pax2ftx} we observe that 
\begin{equation}
    \paxx^2 f(t,x_2) \sim g(t,x_2)R'_2(\log(g(t,x_2)),\log|\log(g(t,x_2))|)
\end{equation}
for some rational function $R'_2$ and therefore $\paxx^2 f(t,x_2) \rightarrow 0$ as $x_2 \rightarrow 0_+$.

In order to handle higher derivatives we first observe that 
\begin{equation}\label{eq:f0dern}
    f_0^{(n)} = \frac{1}{y^n}R_n(\log(y),\log|\log(y)|)
\end{equation}
for some rational functions $R_n$. The expressions for higher derivatives of $f(t,x_2)$ get more complicated, however there is a common feature for all of them, they all consist of sums of terms of the type 
\begin{equation}\label{eq:generalder}
    \frac{(f_0'')^{\alpha_2}(f_0''')^{\alpha_3}...(f_0^{(l)})^{\alpha_l}}{(1+tf_0'(g(t,x_2)))^k}
\end{equation}
with $k \in \mathbb{N}$ taking values between $n+1$ and $2n-1$, $l \in \mathbb{N}$ such that $l \leq n$, $\alpha_i \in \mathbb{N} \cup \{0\}$, $i = 2, ..., l$ and 
\begin{equation}
    \sum_{i = 2}^l i\alpha_i = k-1.
\end{equation}
This is not difficult to prove by induction, the claim clearly holds for $n = 2$ as can be observed in \eqref{eq:pax2ftx}.
On the other hand deriving terms of the form \eqref{eq:generalder} produces terms of the same form. Indeed, either the derivative acts on the numerator of the fraction \eqref{eq:generalder} and in that case the sum of $i\alpha_i$ increases by one, but due to the derivative of the composed function, $\paxx g(t,x_2)$ multiplies this fraction and therefore $k$ increases by one as well, see \eqref{eq:paxg}. If the derivative acts on the denominator, first $k$ increases by one, then another $f_0''$ appears in the numerator which increases the sum of $i\alpha_i$ by two and $\paxx g(t,x_2)$ multiplies the resulting fraction increasing $k$ by another one, altogether therefore $k$ increases by two and so does the sum $i\alpha_i$.

Combining this observation with \eqref{eq:f0dern} we end up with 
\begin{equation}
    \paxx^n f(t,x_2) \sim g(t,x_2)R'_n(\log(g(t,x_2)),\log|\log(g(t,x_2))|)
\end{equation}
with some rational function $R'_n$ and therefore $\paxx^n f(t,x_2) \rightarrow 0$ as $x_2 \rightarrow 0_+$. The proof is finished.

\end{proof}

In our further construction we actually need a slightly more general observation provided by the next lemma.

\begin{lemma}\label{l:f0h}
Let $h_-(x_2) \in C^\infty((-\infty,0)) \cap C^1((-\infty,0])$ and $h_+(x_2) \in C^\infty((0,\infty)) \cap C^1([0,\infty))$ be monotone functions such that $h_\pm(0) = 0$, $h_\pm'(0) = K_\pm$ for some $-\infty < K_- < 0 < K_+ < \infty$ and there exists a constant $C > 0$ such that for $|x_2|$ small enough
\begin{equation}\label{eq:hass}
    |h_\pm^{(n)}(x_2)| \leq \frac{C}{|x_2|^{n-1}}Q_\pm^n(\sqrt{|\log|x_2||}, \log|\log|x_2||)
\end{equation}
for all $n \geq 2$, where $Q_\pm^n$ are rational functions. Let $f_0$ be defined as in \eqref{eq:f0def} and let
\begin{align}
 f_-(T,x_2) &= a_- f_0(h_-(x_2)) + b_-   \\
 f_+(T,x_2) &= -a_+f_0(h_+(x_2)) + b_+
\end{align}
for $0 \leq |x_2| < \overline{\zeta}$ with $h(\overline{\zeta}) < \frac{\zeta_1}{2}$ and with some $a_\pm > 0$ and $b_\pm = \lambda_\pm \pm \frac{\zeta_2}{T}$. For $|x_2| \geq \overline{\zeta}$ define $f_\pm(T,x_2)$ in such a way that $\lambda_1(T,\cdot)$ defined as in \eqref{eq:BursmoothCWT} is smooth and strictly monotone with strictly monotone first derivatives for $x_2 \in (-\infty,0)$ and for $x_2 \in (0,\infty)$.
Then there exists an initial data $\lambda_1^0 \in C^\infty(\R)$ as in \eqref{eq:BurICsmoothCW} such that the solution $\lambda_1(t,x_2)$ to the Burgers equation \eqref{eq:Bur}-\eqref{eq:BurIC} satisfies $\lambda_1 \in C^\infty((0,T)\times \R)$ and $\lambda_1(T,x_2)$ has the form of \eqref{eq:BursmoothCWT}.
\end{lemma}

\begin{proof}
We follow closely the proof of Lemma \ref{l:f0}, for simplicity we show the case $x_2 > 0$ and we write $h = h_+$. Again we want to prove that the solution $f(t,x_2)$ to the Burgers equation \eqref{eq:BurLemma21} with the initial condition 
\begin{equation}
    f(0,x_2) = f_0(h(x_2))
\end{equation}
has the properties \eqref{eq:prvnider}, \eqref{eq:ntader}. The function $g(t,x_2)$ is now defined as 
\begin{equation}\label{eq:charinverse2}
    x_2 = g(t,x_2) + t f_0(h(g(t,x_2))) 
\end{equation}
and thus
\begin{equation}\label{eq:paxg2}
    \paxx g(t,x_2) = \frac{1}{1+tf_0'(h(g(t,x_2)))h'(g(t,x_2))}.
\end{equation}
The solution is given by 
\begin{equation}
    f(t,x_2) = f_0(h(g(t,x_2)))
\end{equation}
and therefore
\begin{equation}\label{eq:paxftx2}
    \paxx f(t,x_2) = \frac{f_0'(h(g(t,x_2)))h'(g(t,x_2))}{1+tf_0'(h(g(t,x_2)))h'(g(t,x_2))}.
\end{equation}
The assumptions on $h$ then ensure that \eqref{eq:prvnider} holds.

The formula for the second derivative of $f$ is 
\begin{equation}\label{eq:pax2ftx2}
    \paxx^2 f(t,x_2) = \frac{f_0''(h(g(t,x_2)))(h'(g(t,x_2)))^2 + f_0'(h(g(t,x_2)))h''(g(t,x_2))}{(1+tf_0'(h(g(t,x_2)))h'(g(t,x_2)))^3},
\end{equation}
which implies 
\begin{align*}
    &\paxx^2 f(t,x_2)  \\ 
    &\quad \sim g(t,x_2) R_2'(\log(h(g(t,x_2))),\log|\log(h(g(t,x_2)))|)Q^2_+(\sqrt{|\log(g(t,x_2))|},\log|\log(g(t,x_2))|)
\end{align*}
for some rational functions $R_2'$ and $Q^2_+$ and therefore \eqref{eq:ntader} holds for $n = 2$.

The argument for higher derivatives is similar as in the proof of Lemma \ref{l:f0}, due to the behaviour of $f_0$ given by \eqref{eq:f0dern} and the assumption \eqref{eq:hass} we end up with 
\begin{align*}
    &\paxx^{(n)} f(t,x_2) \\ 
    &\quad \sim g(t,x_2) R_n'(\log(h(g(t,x_2))),\log|\log(h(g(t,x_2)))|)Q^n_+(\sqrt{|\log(g(t,x_2))|},\log|\log(g(t,x_2))|),
\end{align*}
which finishes the proof by imitating \eqref{eq:ntader} for all $n > 2$.
\end{proof}

Choosing the initial data \eqref{eq:EulerIC} in such a way that $w_1$ defined in \eqref{eq:RiemInv} is constant in $\R$ and $\lambda_1$ defined in \eqref{eq:CharSpeed} takes the form \eqref{eq:BurICsmoothCW} we obtain a smooth solution to the Euler system \eqref{eq:EulerCE}-\eqref{eq:EulerIC} on the time interval $(0,T)$ in the form $w_1 \equiv const$ on $[0,T]\times \R$ and $\lambda_1$ given by \eqref{eq:BursmoothCW}, \eqref{eq:BursmoothCWT}. Such solution satisfies the energy equality \eqref{eq:EulerEnergyIneq} with an equality sign.

The main part of the proof of Theorem \ref{t:main} then consists of proving that the solution on the time interval $[0,T]$ can be prolonged in a non-unique way. In particular we will take values of $(\rho, v)$ at time $t = T$ as the initial data and prove the existence of infinitely many admissible weak solutions starting from this initial data.

\section{Subsolutions}\label{s:ss}

Our aim in this section is to properly define admissible subsolutions. Unlike in previous works on the construction of non--unique admissible solutions, here it is not enough to work with piecewise constant subsolutions. First, we recall the crucial lemma based on the convex integration technique developed for the incompressible Euler equations which allows to pass from a single subsolution to infinitely many solutions, see \cite[Lemma 3.7]{ChiDelKre15}. In what follows, the symbol $\Sz$ stands for the space of symmetric $2\times 2$ matrices with zero trace and $\id$ denotes the identity matrix.

\begin{lemma}
Let $(\tilde{v},\tilde{u}) \in \R^2 \times \Sz$ and let $C > 0$ be such that $\tilde{v}\otimes\tilde{v} - \tilde{u} < \frac C2 \id$. For any open set $\Omega \subset \R \times \R^2$ there are infinitely many maps $(\underline{v}, \underline{u}) \in L^\infty(\R\times \R^2, \R^2 \times \Sz)$ such that
\begin{enumerate}
\item[(i)] $\underline{v}$ and $\underline{u}$ vanish identically outside $\Omega$
\item[(ii)] $\Div \underline{v} = 0$ and $\pat \underline{v} + \Div \underline{u} = 0$
\item[(iii)] $(\tilde{v} + \underline{v}) \otimes (\tilde{v} + \underline{v}) - (\tilde{u} + \underline{u}) = \frac C2 \id$ a.e. in $\Omega$.
\end{enumerate}
\end{lemma}

We need to generalize this lemma in order to be able to work with piecewise continuous subsolutions for the compressible Euler equations instead of piecewise constant ones. We can prove the following Lemma.

\begin{lemma}\label{l:CI}
Let $\Omega \subset \R^2 \times \R$ be an open set. Let $(\tilde{\rho},\tilde{m},\tilde{u},\tilde{C}) \in C(\Omega,\R^+ \times \R^2 \times \Sz \times \R^+)$ such that 
\begin{equation}
\frac{\tilde{m}\otimes\tilde{m}}{\tilde{\rho}} - \tilde{u} < \frac {\tilde{C}}{2} \id
\end{equation} 
pointwise in $\Omega$ in the sense of positive definiteness. Let us assume furthermore that $\tilde{C}$ is a bounded function on $\Omega$. Then, there are infinitely many maps $(\underline{m},\underline{u}) \in L^\infty(\R\times \R^2, \R^2 \times \Sz)$ such that
\begin{enumerate}
\item[(i)] $\underline{m}$ and $\underline{u}$ vanish identically outside $\Omega$.
\item[(ii)] 
\begin{align}
\Div \underline{m} &= 0 \label{eq:INCE}\\
\pat \underline{m} + \Div \underline{u} &= 0 \label{eq:INME}.
\end{align}
\item[(iii)] It holds 
\begin{equation}\label{eq:SSeq}
\frac{(\tilde{m} + \underline{m})\otimes(\tilde{m} + \underline{m})}{\tilde{\rho}}  - (\tilde{u} + \underline{u}) = \frac {\tilde{C}}{2} \id
\end{equation}
a.e. in $\Omega$
\end{enumerate}
\end{lemma}

In order to prove Lemma \ref{l:CI}, we follow the strategy outlined in \cite{ChiDelKre15} for the proof of Lemma 3.7 therein, but we need to borrow some ingredients also from \cite{Ch}.

\subsection{Proof of Lemma \ref{l:CI}} We define $X_0$ to be the space of 
$(\underline{m}, \underline{u})\in C^\infty_c (\Omega, \R^2\times
\Sym_0^{2\times 2})$ which satisfy (ii) and the pointwise inequality 
$\frac{(\tilde{m} + \underline{m})\otimes(\tilde{m} + \underline{m})}{\tilde{\rho}}  - (\tilde{u} + \underline{u}) < \frac {\tilde{C}}{2} \id$.
We then take the closure $X$ of $X_0$ in the $L^\infty$ weak$^\star$ topology.
Since $\tilde{C}$ is a bounded continuous function on $\Omega$, $X$ is a bounded (weakly$^\star$)
closed subset of $L^\infty$. This implies that the $L^\infty$ weak$^\star$ topology is metrizable on $X$, thus producing a complete metric space $(X,d)$. 
Observe that any element in $X$ satisfies (i) and (ii). Our aim is
to show that on a residual set (in the sense of Baire category) (iii) holds. 
Following \cite{ChiDelKre15}, we define for any $N\in \N\setminus \{0\}$ the map $I_N$ as follows: to $(\underline{m}, \underline{u})$ we associate
the corresponding restrictions of these maps to $(-N,N)\times B_N (0) $. We then consider $I_N$ as a map from $(X,d)$ to 
$Y$, where $Y$ is the space $L^\infty ((-N,N)\times B_N (0), \R^2\times \Sym_0^{2\times 2}$) endowed with the {\em strong}
$L^2$ topology. Arguing as in \cite[Lemma 4.5]{DLSz1} it is easily seen that $I_N$ is a Baire-1 map and hence, from
a classical theorem in Baire category, its points of continuity form a residual set in $X$. We claim that
\begin{itemize}
 \item[(CLAIM)] if $(\underline{m},\underline{u})$ is a point of continuity of $I_N$, then (iii) holds a.e. on $(-N,N)\times B_N (0)$.
\end{itemize}
(CLAIM) implies then (iii) for those maps at which {\em all} $I_N$ are continuous (which is also a residual set). 
The proof of (CLAIM) is achieved as in \cite[Lemma 4.6]{DLSz1} showing that:
\begin{itemize}
\item[(Property)] If $(\underline{m}, \underline{u})\in X_0$, then  
there is a sequence
$(m_k, u_k)\subset X_0$ converging weakly$^\star$ to $(\underline{m},\underline{u})$ for which 
\[
\liminf_k \|\tilde{m} + m_k\|_{L^2 (\Gamma)} \geq 
\|\tilde{m} + \underline{m}\|^2_{L^2 (\Gamma)} + \beta \left( \int_\Gamma \tilde{C}\tilde{\rho}  \dxdt - \|\tilde{m} + 
\underline{m}\|^2_{L^2 (\Gamma)}\right)^2\, ,
\]
where $\Gamma= (-N,N)\times B_N (0)$ and $\beta$ depends only on $\Gamma$. 
\end{itemize}
Indeed
assuming that (Property) holds, fix then a point
$(\underline{m}, \underline{u})\in X$ where $I_N$ is continuous and assume by contradiction
that (iii) does not hold on $\Gamma$. By definition of $X$ there is a sequence 
$(\underline{m}_k, \underline{u}_k)\subset X_0$ converging weakly$^\star$ to 
$(\underline{m}, \underline{u})$. Since the latter
is a point of continuity for $I_N$, we then have that $\underline{m}_k \to \underline{m}$ strongly in $L^2 (\Gamma)$. 
We apply (Property) to each $(\underline{m}_k, \underline{u}_k)$ and find a sequence $\{(m_{k,j}, u_{k,j})\}$ such that
\[
\liminf_j \|\tilde{m} + m_{k,j}\|_{L^2 (\Gamma)} \geq 
\|\tilde{m} + \underline{m}_k\|^2_{L^2 (\Gamma)} + \beta \left( \int_\Gamma \tilde{C}\tilde{\rho}  \dxdt - \|\tilde{m} + 
\underline{m}_k\|^2_{L^2 (\Gamma)}\right)^2\, 
\]
and $(m_{k,j}, u_{k,j})\rightharpoonup^\star (\underline{m}_k, \underline{u}_k)$.
A standard diagonal argument then allows to conclude the existence of a sequence $(m_{k, j(k)}, u_{k, j(k)})$
which converges weakly$^\star$ to $(\underline{m}, \underline{u})$ and such that
\[
 \liminf_k \|\tilde{m} + m_{k,j (k)}\|_{L^2 (\Gamma)} \geq 
\|\tilde{m} + \underline{m}\|^2_{L^2 (\Gamma)} + \beta \left(\int_\Gamma \tilde{C}\tilde{\rho}  \dxdt - \|\tilde{m} + 
\underline{m}\|^2_{L^2 (\Gamma)}\right)^2 > \|\tilde{m} + \underline{m}\|^2_{L^2 (\Gamma)}\, .
\]
However this contradicts the assumption that $(\underline{m}, \underline{u})$ is a point of continuity for $I_N$.
Thus, everything is reduced to proving (Property), i.e. to constructing the sequence
$(m_k, u_k)$. For the proof of Property we refer the reader to \cite[Lemma 4.5]{Ch} whose statement is analogous; the only two differences are that on the one hand the role of $\tilde{C}$ is played in \cite{Ch} by the function $\chi$  which depends on time only and on the other hand here $\tilde{\rho}$ depends not only on the space variable as in \cite{Ch}, but also on time. However these differences are not substantial for the proof of \cite[Lemma 4.5]{Ch} (and hence of (Property)), which can be carried out similarly under the current setting.

\subsection{From subsolutions to solutions}
As we stated in the previous section, we assume here that the initial data for the isentropic Euler equations take the form

\begin{equation}\label{eq:ICGR}
(\rho^0,v^0) = \left\{\begin{split}
&(\rho^0_-,v^0_-) \qquad \text{ for } x_2 < 0 \\
&(\rho^0_+,v^0_+) \qquad \text{ for } x_2 > 0,
\end{split} \right.
\end{equation}
where $(\rho^0_\pm,v^0_\pm)$ are $C^\infty$ functions on $(-\infty,0)$ and $(0,\infty)$ respectively such that there exists $\zeta_1 > 0$ and constants $(R_\pm, V_\pm) \in \R^+\times \R^2$ such that 
\begin{equation}\label{eq:ICGRProp}
\begin{split}
(\rho^0_-,v^0_-) &= (R_-,V_-) \qquad \text{ for } x_2 < -\zeta_1 \\
(\rho^0_+,v^0_+) &= (R_+,V_+) \qquad \text{ for } x_2 > \zeta_1.
\end{split}
\end{equation}

Moreover we use the momentum formulation of the Euler equations instead of using the velocity, i.e. we introduce $m = \rho v$ and rewrite the Euler system \eqref{eq:EulerCE}-\eqref{eq:EulerIC} to

\begin{align} \label{eq:EEm1}
\pat \rho + \Div m &= 0 \\
\pat m + \Div \left(\frac{m \otimes m}{\rho}\right) + \nabla_x p(\rho) &= 0  \label{eq:EEm2} \\
(\rho,m)(\cdot,0) &= (\rho^0, m^0). \label{eq:EEmI}
\end{align}

The admissibility (or energy) inequality \eqref{eq:EulerEnergyIneq} is rewritten as
\begin{equation}\label{eq:EEin}
\pat \left(\rho e(\rho) + \frac 12\frac{\abs{m}^2}{\rho}\right) + \Div \left(\left(\rho e(\rho) + p(\rho) + \frac 12 \frac{\abs{m}^2}{\rho}\right) \frac{m}{\rho} \right) \leq 0.
\end{equation}

In order to define subsolutions we will work with in this paper we start with a generalized notion of fan partition.

\begin{definition}
A generalized fan partition of $(0,\infty) \times \R^2$ consists of three open sets $P_-,P_1,P_+$ such that
\begin{align}
P_- &= \left\{(t,x): t > 0 \quad \text{ and } \quad x_2 < \widetilde{\nu}_-(t)\right\} \\
P_1 &= \left\{(t,x): t > 0 \quad \text{ and } \quad \widetilde{\nu}_-(t) <  x_2 < \widetilde{\nu}_+(t)\right\} \\
P_+ &= \left\{(t,x): t > 0 \quad \text{ and } \quad \widetilde{\nu}_+(t) < x_2\right\},
\end{align}
where $\widetilde{\nu}_\pm(t)$ is a couple of continuous functions satisfying $\widetilde{\nu}_-(t) < \widetilde{\nu}_+(t)$ for all $t > 0$ and $\widetilde{\nu}_-(0) = \widetilde{\nu}_+(0) = 0.$
\end{definition}

\begin{definition}
A generalized fan subsolution to the compressible Euler equations is a quadruple of piecewise continuous functions $(\overline{\rho},\overline{m},\overline{u},\overline{C}) : (0,\infty)\times \R^2 \rightarrow (\R^+,\R^2,\Sz,\R^+)$ satisfying the following properties.
\begin{itemize}
\item[(i)] There exists a generalized fan partition $P_-,P_1,P_+$ of $ (0,\infty)\times \R^2$ such that 
\begin{equation}
(\overline{\rho},\overline{m},\overline{u},\overline{C}) = \sum_{i=-,1,+} (\rho_i,m_i,u_i,C_i) \chi_{P_i},
\end{equation}
where $(\rho_i,m_i,u_i,C_i)$ are continuous functions on $P_i$, $u_\pm = \frac{m_\pm \otimes m_\pm}{\rho_\pm} - \frac 12 \frac{\abs{m_\pm}^2}{\rho_\pm}\id$, $C_\pm = \frac{\abs{m_\pm}^2}{\rho_\pm}$ and 
\begin{equation}\label{eq:SSIC}
\lim_{t \rightarrow 0^+} (\rho_\pm,m_\pm)(t,x) = (\rho^0_\pm,m^0_\pm)(x) \qquad \text{ for } \pm x > 0.
\end{equation}
\item[(ii)] It holds 
\begin{equation}\label{eq:SSCond}
\frac{m_1\otimes m_1}{\rho_1} - u_1 < \frac{C_1}{2} \id
\end{equation}
pointwise in $P_1$ in the sense of positive definiteness.
\item[(iii)] The quadruple $(\overline{\rho},\overline{m},\overline{u},\overline{C})$ solves the following system of partial differential equations in the sense of distributions
\begin{align}
\pat \overline{\rho} + \Div \overline{m} &= 0 \label{eq:SSCE} \\
\pat \overline{m} + \Div \overline{u} + \nabla_x \left( p(\overline{\rho}) + \frac{\overline{C}}{2} \right) &= 0.\label{eq:SSME}
\end{align}

\end{itemize}
\end{definition}

\begin{definition}
A generalized fan subsolution $(\overline{\rho},\overline{m},\overline{u},\overline{C})$ to the compressible Euler equations is called admissible if it satisfies that
\begin{itemize}
\item[(i)] there exists $K(t) : (0,\infty) \rightarrow \R^+$ such that for all $(x,t) \in P_1$ it holds 
\begin{equation}\label{eq:Kdef}
\left(e(\rho_1) + \frac{p(\rho_1)}{\rho_1} + \frac{C_1}{2\rho_1}\right)(x,t) = K(t),
\end{equation}
\item[(ii)] the following inequality is satisfied in the sense of distributions
\begin{equation}\label{eq:SSIN}
\pat \left(\overline{\rho}e(\overline{\rho}) + \frac{\overline{C}}{2}\right) + \Div\left(\left(\overline{\rho}e(\overline{\rho}) + p(\overline{\rho}) + \frac{\overline{C}}{2}\right)\frac{\overline{m}}{\overline{\rho}}\right) \leq 0.
\end{equation}
\end{itemize}
\end{definition}

The importance of assumptions in the definitions of the admissible generalized fan subsolution is revealed in the following Proposition.
\begin{propo}\label{p:main}
Let $p$ be any $C^1$ function and $(\rho^0,m^0)$ be such that there exists at least one admissible generalized fan subsolution $(\overline{\rho},\overline{m},\overline{u},\overline{C})$ to the Euler equations \eqref{eq:EEm1}--\eqref{eq:EEm2} with initial data \eqref{eq:EEmI}. Then there exists infinitely many bounded admissible weak solutions $(\rho,m)$ to \eqref{eq:EEm1}--\eqref{eq:EEmI} such that $\rho = \overline{\rho}$.
\end{propo}

\begin{proof}
We use Lemma \ref{l:CI} in the region $P_1$ with $(\tilde{\rho},\tilde{m},\tilde{u},\tilde{C}) = (\rho_1,m_1,u_1,C_1)$ and achieve infinitely many maps $(\underline{m},\underline{u})$. We set $\rho = \overline{\rho}$, $m = \overline{m} + \underline{m}$, $u = \overline{u} + \underline{u}$ and $C = \overline{C}$. We need to show that $(\rho,m)$ is an admissible weak solution to the Euler equations \eqref{eq:EEm1}-\eqref{eq:EEmI}.

First of all we observe that due to \eqref{eq:SSeq} we know that in the region $P_1$ it holds
\begin{align}\label{eq:AAA1}
u &= \frac{m\otimes m}{\rho} - \frac 12 \frac{\abs{m}^2}{\rho}\id \\
C &= \frac{\abs{m}^2}{\rho},\label{eq:AAA2}
\end{align}
whereas we already know that \eqref{eq:AAA1}-\eqref{eq:AAA2} hold in regions $P_-$ and $P_+$ due to the definition of the subsolution.

Combining \eqref{eq:INCE} with \eqref{eq:SSCE} we easily observe that $(\rho,m)$ satisfy \eqref{eq:EEm1}. Similarly combining \eqref{eq:INME} with \eqref{eq:SSME} we get
\begin{equation}
\pat m + \Div u + \nabla_x \left( p(\rho) + \frac{C}{2} \right) = 0
\end{equation}
and plugging in \eqref{eq:AAA1} and \eqref{eq:AAA2} we get \eqref{eq:EEm2}. Finally, we observe that using \eqref{eq:Kdef}, \eqref{eq:INCE} and the fact that $\underline{m}$ is supported in $P_1$ we have 
\begin{equation}\label{eq:AAA3}
\Div \left(\left(\rho e(\rho) + p(\rho) + \frac{C}{2}\right) \frac{\underline{m}}{\rho}\right) = \Div \left(\left(\rho_1e(\rho_1) + p(\rho_1) + \frac{C_1}{2}\right) \frac{\underline{m}}{\rho_1}\right) = K(t)\Div \underline{m} = 0.
\end{equation}
Therefore it is enough to sum \eqref{eq:SSIN} with \eqref{eq:AAA3} to obtain \eqref{eq:EEin}.

The proof of Proposition \ref{p:main} is concluded by observing that as $\tau \rightarrow 0+$ the Lebesgue measure of $P_1 \cap \{t = \tau\}$ converges to zero and therefore the attainment of the initial conditions follows from \eqref{eq:SSIC}.
\end{proof}

\section{The existence of a subsolution}\label{s:main}

In this section we study further the admissible generalized fan subsolution which we will look for.

\subsection{Basic outline}

We start with examining some properties of solution to the Burgers equation starting from the initial data of the form we are interested in. More precisely, following the discussion in Section \ref{s:cw} and in particular the form of the solution to the Burgers equation at a time instant when the smooth compression wave collapses \eqref{eq:BursmoothCWT}, we consider now $\lambda^0_1$ such that it is smooth on intervals $(-\infty, 0)$ and $(0,\infty)$ and 
\begin{equation}\label{eq:BurICGR}
\lambda_1^0(x_2) = \left\{\begin{split}
&\lambda_- \qquad \text{ for } x_2 < - \zeta_1 \\
&\widetilde{f}_-(x_2) \qquad \text{ for } - \zeta_1 < x_2 < 0 \\
&\widetilde{f}_+(x_2) \qquad \text{ for } 0 < x_2 <  \zeta_1 \\
&\lambda_+ \qquad \text{ for } x_2 >  \zeta_1
\end{split} \right.
\end{equation}
where $\widetilde{f}_\pm$ are smooth functions such that for $|x_2| < \overline{\zeta}$ with sufficiently small $\overline{\zeta}$ it holds
\begin{equation}\label{eq:Propertiesfpm}
    \begin{split}
        \widetilde{f}_-(x_2) &= a_- f_0(h_-(x_2)) + b_-   \\
 \widetilde{f}_+(x_2) &= -a_+f_0(h_+(x_2)) + b_+.
    \end{split}
\end{equation}
with some $a_\pm > 0$ and $b_\pm = \lambda_\pm \pm \frac{\zeta_2}{T}$, where $f_0$ is given by \eqref{eq:f0def} and $h_\pm$ satisfy the assumptions of Lemma \ref{l:f0h}.


\begin{lemma}\label{l:41}
Let $\lambda^0_1$ has the form \eqref{eq:BurICGR} such that \eqref{eq:Propertiesfpm} holds. Then there exist functions $s_\pm(t)$ such that $s_-(0) = s_+(0) = 0$, $s_-(t) < s_+(t)$ for all $t > 0$ and the admissible solution $\lambda_1(t,x_2)$ to the Burgers equation \eqref{eq:Bur}-\eqref{eq:BurIC} is $C^\infty$ on the sets $x_2 < s_-(t)$ and $x_2 > s_+(t)$.
\end{lemma}
\begin{proof}
Since $\lambda_- > \lambda_+$, the admissible solution to the Burgers equation contains a shock. The shock curve satisfies the Rankine-Hugoniot conditions
\begin{equation}
s(\lambda_{1R} - \lambda_{1L}) = \frac{1}{2}(\lambda_{1R}^2-\lambda_{1L}^2)
\end{equation}
yielding
\begin{equation}
s = \frac{1}{2} (\lambda_{1R}+\lambda_{1L}),
\end{equation}
where $s$ is the shock speed and $\lambda_{1R}, \lambda_{1L}$ are values on the right and on the left side of the shock respectively. Even though the precise position of the shock curve in spacetime is obviously related to the functions $\widetilde{f}_\pm$, it follows by the continuity of the related quantities 
that the shock speed will always belong to the interval
\begin{equation}
s \in \left(\frac 12 (\lambda_++\lambda_- - \frac{\zeta_2}{T}), \frac 12 (\lambda_++\lambda_- + \frac{\zeta_2}{T})\right)
\end{equation}
and therefore it is enough to set
\begin{equation}
s_\pm(t) = \frac 12 (\lambda_++\lambda_- \pm \frac{\zeta_2}{T})t.
\end{equation}
\end{proof}

Our goal is to find an admissible generalized fan subsolution to the compressible Euler equations with the initial data generated by a smooth compression wave. We will search for a subsolution with the property that $\rho_1,m_1,u_1$ and $C_1$ do not depend on $x_1$ and are therefore functions only of $t$ and $x_2$. In particular we have to make the following steps
\begin{itemize}
\item[(i)] Find functions $\widetilde{\nu}_-(t)$, $\widetilde{\nu}_+(t)$ describing the generalized fan partition.
\item[(ii)] Construct solutions to the Euler equations in the regions $P_-$, $P_+$ with the initial data \eqref{eq:ICGR}-\eqref{eq:ICGRProp} generated by a smooth compression wave. These $C^\infty$ solutions will be constructed by the method of characteristics for the Burgers equation \eqref{eq:Bur} for $\lambda_1$ in regions $P_-$ and $P_+$ while keeping $w_1$ constant. This will be possible provided $\widetilde{\nu}_- < s_-$ and $\widetilde{\nu}_+ > s_+$ for all $t > 0$.
\item[(iii)] In order to the equations \eqref{eq:SSCE}-\eqref{eq:SSME} and the inequality \eqref{eq:SSIN} to be satisfied on the discontinuities given by $x_2 = \widetilde{\nu}_\pm(t)$, the appropriate Rankine-Hugoniot conditions need to be satisfied on these two interfaces.
\item[(iv)] We need to ensure that in the region $P_1$ the equations \eqref{eq:SSCE}-\eqref{eq:SSME} are satisfied, the inequalities \eqref{eq:SSCond}, \eqref{eq:SSIN} are satisfied and the condition \eqref{eq:Kdef} holds.
\end{itemize}

In the following we will investigate the points above. The actual construction is to some extent reverse; by first defining $\rho^0,m^0$ close to a certain piecewise constant Riemann state and then deducing from these functions the fan partition. This construction can be found at the end of the paper.

\subsection{A general set of conditions}

We set 
\begin{align}
m_1 &= (\alpha,\beta) \\
u_1 &= \left(
\begin{array}{cc}
\gamma_1 & \gamma_2 \\
\gamma_2 & - \gamma_1
\end{array}\right) \\
\nu_\pm(t) &= \frac{\d}{\d t} \widetilde{\nu}_\pm(t)
\end{align}
and we introduce the following notation
\begin{align}
(\rho_\pm^\nu, m_{\pm 2}^\nu)(t) &:= (\rho_\pm,m_{\pm 2})(t,\widetilde{\nu}_\pm(t)) \\
f_L(t) &:= f(t,\widetilde{\nu}_-(t)) \\
f_R(t) &:= f(t,\widetilde{\nu}_+(t))
\end{align}
for any function $f$ defined in $P_1$.

Let us now write down the set of Rankine-Hugoniot conditions described in the point (iii) above. For each time $t > 0$ we have
\begin{itemize}
\item Rankine-Hugoniot conditions on the left interface:
\begin{align}
&\nu_- (\rho_-^\nu - \rho_{1L}) \, =\,  m_{-2}^\nu - \beta_L \label{eq:cont_left}  \\
&\nu_- (m_{-1}^\nu - \alpha_L) \, = \, \frac{m_{-1}^\nu m_{-2}^\nu}{\rho_-^\nu} - \gamma_{2L}  \label{eq:mom_1_left}\\
&\nu_- (m_{-2}^\nu - \beta_L) \, = \,  
\frac{(m_{-2}^\nu)^2}{\rho_-^\nu}  + \gamma_{1L} + p (\rho_-^\nu)-p (\rho_{1L}) - \frac{C_{1L}}{2}\, ;\label{eq:mom_2_left}
\end{align}
\item Rankine-Hugoniot conditions on the right interface:
\begin{align}
&\nu_+ (\rho_{1R}-\rho_+^\nu ) \, =\,  \beta_R - m_{+2}^\nu \label{eq:cont_right}\\
&\nu_+ (\alpha_R - m_{+1}^\nu) \, = \, \gamma_{2R} - \frac{m_{+1}^\nu m_{+2}^\nu}{\rho_+^\nu} \label{eq:mom_1_right}\\
&\nu_+ (\beta_R - m_{+2}^\nu) \, = \, - \gamma_{1R} - \frac{(m_{+2}^\nu)^2}{\rho_+^\nu}  +p (\rho_{1R}) -p (\rho_+^\nu) 
+ \frac{C_{1R}}{2}\, ;\label{eq:mom_2_right}
\end{align}
\item admissibility condition on the left interface:
\begin{align}
& \nu_-(\rho_-^\nu e(\rho_-^\nu)- \rho_{1L} e( \rho_{1L}))+\nu_- 
\left(\frac{\abs{m_-^\nu}^2}{2\rho_-^\nu}- \frac{C_{1L}}{2}\right)\nonumber\\
\leq & \left[(\rho_-^\nu e(\rho_-^\nu)+ p(\rho_-^\nu)) \frac{m_{-2}^\nu}{\rho_-^\nu}- 
( \rho_{1L} e( \rho_{1L})+ p(\rho_{1L}))\frac{\beta_L}{\rho_{1L}} \right] 
+ \left( m_{-2}^\nu \frac{\abs{m_-^\nu}^2}{2(\rho_-^\nu)^2}- \beta_L \frac{C_{1L}}{2\rho_{1L}}\right)\, ;\label{eq:E_left}
\end{align}
\item admissibility condition on the right interface:
\begin{align}
&\nu_+(\rho_{1R} e( \rho_{1R})- \rho_+^\nu e(\rho_+^\nu))+\nu_+ 
\left( \frac{C_{1R}}{2}- \frac{\abs{m_+^\nu}^2}{2\rho_+^\nu}\right)\nonumber\\
\leq &\left[ ( \rho_{1R} e( \rho_{1R})+ p(\rho_{1R})) \frac{\beta_R}{\rho_{1R}}- (\rho_+^\nu e(\rho_+^\nu)+ p(\rho_+^\nu)) \frac{m_{+2}^\nu}{\rho_+^\nu}\right] 
+ \left( \beta_R \frac{C_{1R}}{2\rho_{1R}}- m_{+2}^\nu \frac{\abs{m_+^\nu}^2}{2(\rho_+^\nu)^2}\right)\, .\label{eq:E_right}
\end{align}
\end{itemize}

Finally we write down the conditions mentioned in the point (iv) above. We recall that we search for subsolution independent of $x_1$, therefore we have from \eqref{eq:SSCE}-\eqref{eq:SSME}
\begin{align}
\pat \rho_1 + \paxx \beta &= 0 \label{eq:SSCE2} \\
\pat \alpha + \paxx \gamma_2 &= 0 \label{eq:SSME20} \\
\pat \beta + \paxx \left( p(\rho_1) + \frac{C_1}{2} - \gamma_1 \right) &= 0,\label{eq:SSME2}
\end{align}
the subsolution condition \eqref{eq:SSCond} transforms to
\begin{align}
 &\alpha^2 +\beta^2 < \rho_1 C_1 \label{eq:sub_trace}\\
& \left( \frac{C_1}{2} -\frac{\alpha^2}{\rho_1} +\gamma_1 \right) \left( \frac{C_1}{2} -\frac{\beta^2}{\rho_1} -\gamma_1 \right) - 
\left( \gamma_2 - \frac{\alpha \beta}{\rho_1} \right)^2 >0,\label{eq:sub_det}
\end{align}
we rewrite \eqref{eq:Kdef} as
\begin{equation}\label{eq:Kdef2}
\left(\rho_1 e(\rho_1) + p(\rho_1) + \frac{C_1}{2}\right)(t,x_2) = \rho_1(t,x_2) K(t)
\end{equation}
and we can formulate \eqref{eq:SSIN} as 
\begin{equation}\label{eq:SSIN2}
\pat \left(\rho_1 e(\rho_1) + \frac{C_1}{2}\right) + \paxx \left(\left(\rho_1 e(\rho_1) + p(\rho_1) + \frac{C_1}{2}\right)\frac{\beta}{\rho_1}\right) \leq 0.
\end{equation}

\subsection{Simplifications and ansatz}

We continue with several observations which were already introduced in previous works, for more details see \cite[Section 4]{ChiKre14}. First, our initial data $(\rho^0,m^0)$ will be chosen in such a way that $m_{-1}^0 = v_{-1}^0 = m_{+1}^0 = v_{+1}^0 \equiv 0$. This directly implies $m_{-1} = m_{+1} \equiv 0$ and we will therefore look for subsolution with similar property, namely $\alpha \equiv 0$ and $\gamma_2 \equiv 0$. This choice implies that \eqref{eq:mom_1_left}, \eqref{eq:mom_1_right} and \eqref{eq:SSME20} are trivially satisfied and \eqref{eq:sub_trace}-\eqref{eq:sub_det} simplify to
\begin{align}
 &\beta^2 < \rho_1 C_1 \label{eq:sub_trace2}\\
& \left( \frac{C_1}{2} +\gamma_1 \right) \left( \frac{C_1}{2} -\frac{\beta^2}{\rho_1} -\gamma_1 \right) >0.\label{eq:sub_det2}
\end{align}
The necessary condition for \eqref{eq:sub_trace2}-\eqref{eq:sub_det2} to be satisfied is $\frac{C_1}{2} - \gamma_1 > \frac{\beta^2}{\rho_1}$ (see \cite[Lemma 4.3]{ChiKre14}), which motivates us to introduce 
\begin{align}\label{eq:ep1def}
\ep_1 &:= \frac{C_1}{2} - \gamma_1 - \frac{\beta^2}{\rho_1}\\
\ep_2 &:= C_1 - \frac{\beta^2}{\rho_1} - \ep_1\label{eq:ep2def}
\end{align}
and \eqref{eq:sub_trace2}-\eqref{eq:sub_det2} further simplify to
\begin{align}
 &\ep_1 > 0 \label{eq:sub_trace3}\\
&\ep_2 > 0.\label{eq:sub_det3}
\end{align}

Now we rewrite the remaining equations and inequalities \eqref{eq:cont_left}-\eqref{eq:SSIN2} using $\ep_1$, $\ep_2$ instead of $C_1$ and $\gamma_1$ and plugging in also our choice of the pressure law $p(\rho) = \rho^2$, which yields $e(\rho) = \rho$. We get 
\begin{itemize}
\item Rankine-Hugoniot conditions on the left interface:
\begin{align}
&\nu_- (\rho_-^\nu - \rho_{1L}) \, =\,  m_{-2}^\nu - \beta_L \label{eq:cont_left2}  \\
&\nu_- (m_{-2}^\nu - \beta_L) \, = \,  
\frac{(m_{-2}^\nu)^2}{\rho_-^\nu} + (\rho_-^\nu)^2 - \frac{\beta_L^2}{\rho_{1L}} - \rho_{1L}^2 - \ep_{1L}\, ;\label{eq:mom_2_left2}
\end{align}
\item Rankine-Hugoniot conditions on the right interface:
\begin{align}
&\nu_+ (\rho_{1R}-\rho_+^\nu ) \, =\,  \beta_R - m_{+2}^\nu \label{eq:cont_right2}\\
&\nu_+ (\beta_R - m_{+2}^\nu) \, = \,  \frac{\beta_R^2}{\rho_{1R}} + \rho_{1R}^2 + \ep_{1R}
- \frac{(m_{+2}^\nu)^2}{\rho_+^\nu}  - (\rho_+^\nu)^2 \, ;\label{eq:mom_2_right2}
\end{align}
\item admissibility condition on the left interface:
\begin{align}
& \nu_-((\rho_-^\nu)^2 - \rho_{1L}^2 )+\nu_- 
\left(\frac{(m_{-2}^\nu)^2}{2\rho_-^\nu}- \frac{\beta_L^2}{2\rho_{1L}} - \frac{\ep_{1L}+\ep_{2L}}{2}\right)\nonumber\\
\leq & \left(2\rho_-^\nu m_{-2}^\nu - 
2\rho_{1L}\beta_L \right) 
+ \left( \frac{(m_{-2}^\nu)^3}{2(\rho_-^\nu)^2}- \frac{\beta_{L}^3}{2\rho_{1L}^2} - \frac{\beta_L(\ep_{1L}+\ep_{2L})}{2\rho_{1L}}\right)\, ;\label{eq:E_left2}
\end{align}
\item admissibility condition on the right interface:
\begin{align}
&\nu_+(\rho_{1R}^2- (\rho_+^\nu)^2)+\nu_+ 
\left( \frac{\beta_R^2}{2\rho_{1R}} + \frac{\ep_{1R}+\ep_{2R}}{2}- \frac{(m_{+R}^\nu)^2}{2\rho_+^\nu}\right)\nonumber\\
\leq &\left( 2\rho_{1R}\beta_R - 2\rho_+^\nu m_{+2}^\nu\right) 
+ \left( \frac{\beta_{R}^3}{2\rho_{1R}^2} + \frac{\beta_R(\ep_{1R}+\ep_{2R})}{2\rho_{1R}} - \frac{(m_{+2}^\nu)^3}{2(\rho_+^\nu)^2}\right)\, ;\label{eq:E_right2}
\end{align}
\item differential equations in $P_1$:
\begin{align}
\pat \rho_1 + \paxx \beta &= 0 \label{eq:SSCE3} \\
\pat \beta + \paxx \left( \rho_1^2 + \frac{\beta^2}{\rho_1} + \ep_1 \right) &= 0\, ;\label{eq:SSME3}
\end{align}
\item admissibility in $P_1$:
\begin{align}\label{eq:Kdef3}
\left(2\rho_1^2 + \frac{1}{2}\left(\frac{\beta^2}{\rho_1}+\ep_1+\ep_2\right)\right)(t,x_2) &= \rho_1(t,x_2) K(t) \\
\label{eq:SSIN3}
\pat \left(\rho_1^2 + \frac{1}{2}\left(\frac{\beta^2}{\rho_1}+\ep_1+\ep_2\right)\right) + K(t)\paxx \beta &\leq 0\, .
\end{align}
\end{itemize}
Using \eqref{eq:Kdef3} we moreover rewrite \eqref{eq:SSME3} and \eqref{eq:SSIN3} to
\begin{align}
\pat \beta + \paxx \left( 2\rho_1 K - 3\rho_1^2 - \ep_2 \right) &= 0\, ;\label{eq:SSME4} \\
\pat \left(\rho_1 K - \rho_1^2 \right) + K(t)\paxx \beta &\leq 0\, .\label{eq:SSIN4}
\end{align}

At this point we introduce the following ansatz for the part of the subsolution supported in the region $P_1$. We will look for the admissible generalized fan subsolution having the following properties.
\begin{itemize}
\item $\rho_1$ is constant in the whole $P_1$;
\item $K$ is constant in the whole $P_1$;
\item $\beta$ depends only on $t$ and is independent of $x_2$.
\end{itemize}
With this ansatz, \eqref{eq:SSCE3} is satisfied trivially, \eqref{eq:SSIN4} is satisfied trivially as an equation and \eqref{eq:SSME4} simplifies into
\begin{equation}
\pat \beta - \paxx \ep_2 = 0\, .\label{eq:SSME5}
\end{equation}
Since we assume that $\beta$ is independent of $x_2$, the same has to hold also for $\paxx \ep_2$, which implies that $\ep_2$ has to be linear in the $x_2$-variable and \eqref{eq:SSME5} therefore becomes
\begin{equation}
\pat \beta = \frac{\ep_{2R}-\ep_{2L}}{l},\label{eq:SSME6}
\end{equation}
where
\begin{equation}\label{eq:ldef}
l = l(t) = \widetilde{\nu}_+(t) - \widetilde{\nu}_-(t) = \int_0^t (\nu_+(t) - \nu_-(t)) \dt > 0.
\end{equation}

Moreover we rewrite \eqref{eq:cont_left2}-\eqref{eq:E_right2} once again, this time using \eqref{eq:Kdef3} to replace $\ep_1$ by $K$. Note that as a consequence of our ansatz, $\ep_2$ remains in this set of equaitons and inequalities the only function defined in $P_1$ with different values of $\ep_{2L}$ and $\ep_{2R}$. In all other functions we can skip the subscripts $L$ and $R$, since these functions are constant for fixed $t$. We obtain
\begin{itemize}
\item Rankine-Hugoniot conditions on the left interface:
\begin{align}
&\nu_- (\rho_-^\nu - \rho_{1}) \, =\,  m_{-2}^\nu - \beta \label{eq:cont_left3}  \\
&\nu_- (m_{-2}^\nu - \beta) \, = \,  
\frac{(m_{-2}^\nu)^2}{\rho_-^\nu} + (\rho_-^\nu)^2 - 2\rho_1 K + 3\rho_1^2 + \ep_{2L}\, ;\label{eq:mom_2_left3}
\end{align}
\item Rankine-Hugoniot conditions on the right interface:
\begin{align}
&\nu_+ (\rho_{1}-\rho_+^\nu ) \, =\,  \beta - m_{+2}^\nu \label{eq:cont_right3}\\
&\nu_+ (\beta - m_{+2}^\nu) \, = \,  2\rho_1 K - 3\rho_{1}^2 - \ep_{2R}
- \frac{(m_{+2}^\nu)^2}{\rho_+^\nu}  - (\rho_+^\nu)^2 \, ;\label{eq:mom_2_right3}
\end{align}
\item Admissibility condition on the left interface:
\begin{equation}
\nu_- \left((\rho_-^\nu)^2 + \rho_{1}^2 + \frac{(m_{-2}^\nu)^2}{2\rho_-^\nu}- \rho_1 K\right)\leq 2\rho_-^\nu m_{-2}^\nu +  \frac{(m_{-2}^\nu)^3}{2(\rho_-^\nu)^2}- \beta K\, ;\label{eq:E_left3}
\end{equation}
\item Admissibility condition on the right interface:
\begin{equation}
\nu_+ \left(\rho_1 K - \rho_1^2 - (\rho_+^\nu)^2 - \frac{(m_{+R}^\nu)^2}{2\rho_+^\nu}\right)
\leq \beta K - 2\rho_+^\nu m_{+2}^\nu - \frac{(m_{+2}^\nu)^3}{2(\rho_+^\nu)^2}\, .\label{eq:E_right3}
\end{equation}
\end{itemize}

\subsection{Solution for Riemann initial data}\label{ss:Riem}

As it was already shown in \cite{ChiDelKre15}, there exist infinitely many admissible weak solutions starting from Riemann initial data generated by a compression wave. We will use the same result as a starting point, however we present a proof of existence of little bit different solutions than those proved to exist in \cite{ChiDelKre15}. Here we assume that the initial data have the form
\begin{equation}\label{eq:ICRiemann}
(\rho^0,v^0) = \left\{\begin{split}
&(R_-,V_-) \qquad \text{ for } x_2 < 0 \\
&(R_+,V_+) \qquad \text{ for } x_2 > 0.
\end{split} \right.
\end{equation}

\begin{lemma}\label{l:Riemann}
Let $p(\rho) = \rho^2$. Let $(\rho^0,v^0)$ have the form \eqref{eq:ICRiemann} with $R_- = 1$, $R_+ = 4$, $V_- = (0,\sqrt{8})$, $V_+ = (0,0)$. Then there exist infinitely many bounded admissible weak solutions to the Euler equations \eqref{eq:EulerCE}-\eqref{eq:EulerIC}.
\end{lemma}
\begin{proof}
First, it is easy to observe, that such initial data are generated by a compression wave, see \cite[Lemma 6.1]{ChiDelKre15}. We will show that there exists a piecewise constant admissible fan subsolution which yields the existence of infinitely many bounded admissible weak solutions using Proposition \ref{p:main}.

Such subsolution is given by the following set of numbers
\begin{equation}\label{eq:subsolRiem}
\begin{array}{lll}\displaystyle
\nu_- = \frac{-\sqrt{8}-\sqrt{26}}{3} & \displaystyle \nu_+ = \frac{\sqrt{26}-\sqrt{32}}{6} & \\
\rho_1 = 2 & \displaystyle \beta = \frac{\sqrt{32}-\sqrt{26}}{3} & \\
\displaystyle \ep_1 = \frac{50+16\sqrt{13}}{9} & \ep_2 = 1 & \displaystyle K = \frac{58+2\sqrt{13}}{9}.
\end{array}
\end{equation}

It is a matter of computation to check that equations \eqref{eq:cont_left3}-\eqref{eq:mom_2_right3} are satisfied and inequalities \eqref{eq:E_left3}-\eqref{eq:E_right3} are satisfied as strict inequalities. The equation \eqref{eq:SSME6} is of course satisfied trivially since we work now with piecewise constant functions.
\end{proof}

\subsection{Generalization to nonconstant $\ep_2$}\label{ss:Gen}

Let us recall again that the initial data we consider at this point in the proof have the form \eqref{eq:ICGR}-\eqref{eq:ICGRProp} and are generated by a smooth compression wave which approximates the compression wave from Lemma \ref{l:Riemann}. This means that the Riemann invariant $w_1$ defined in \eqref{eq:RiemInv} is constant (the value of this constant is actually $w_1 = 4\sqrt{2}$) and the initial characteristic speed $\lambda_1^0(x_2)$ defined in \eqref{eq:CharSpeed} has the form \eqref{eq:BurICGR}-\eqref{eq:Propertiesfpm}.

In particular taking $\zeta_2 > 0$ small enough, we know that the values of $(\rho^0_\pm,m^0_\pm)$ are close to $(R_\pm,M_\pm)$\footnote{We denote $M_\pm = R_\pm V_\pm$.}. Therefore the same property holds also for $(\rho_\pm^\nu,m_\pm^\nu)$, since the values of $(\rho_\pm,m_\pm)$ are propagated along characteristics of the Burgers equation \eqref{eq:Bur}-\eqref{eq:BurIC}.

We proceed with the construction of the subsolution as follows. We fix the values
\begin{equation}\displaystyle
\rho_1 = 2 \qquad \qquad K = \frac{58+2\sqrt{13}}{9}.\label{eq:rKchoice}
\end{equation}
For fixed $t > 0$ we solve the equations \eqref{eq:cont_left3}-\eqref{eq:mom_2_right3} to obtain $\nu_\pm$, $\beta$ and $\ep_{2L}$ as a functions of the parameter 
\[
\epd := \ep_{2L}-\ep_{2R}.
\]
We denote
\begin{align}
R &= \rho_-^\nu - \rho_+^\nu \label{eq:Rdef} \\
\overline{R} &= R_- - R_+ = -3 \\
A &= m_{-2}^\nu - m_{+2}^\nu  \label{eq:Adef} \\
\overline{A} &= M_{-2} - M_{+2} = \sqrt{8} \\
H &= \frac{(m_{-2}^\nu)^2}{\rho_-^\nu} - \frac{(m_{+2}^\nu)^2}{\rho_+^\nu} + (\rho_-^\nu)^2 - (\rho_+^\nu)^2 \\
B &= A^2 - RH = \rho_-^\nu\rho_+^\nu\left(\frac{m_{-2}^\nu}{\rho_-^\nu} - \frac{m_{+2}^\nu}{\rho_+^\nu}\right)^2 - (\rho_-^\nu - \rho_+^\nu)((\rho_-^\nu)^2-(\rho_+^\nu)^2) \label{eq:Bdef}\\
\overline{B} &= R_-R_+\left(\frac{M_{-2}}{R_-} - \frac{M_{+2}}{R_+}\right)^2 - (R_- - R_+)(R_-^2-R_+^2) = -13.
\end{align}
and we emphasize that for the initial conditions under consideration $(R,A,B)$ are functions of $t$, they are independent of the parameter $\epd$ and they take values in the neighborhood of $(\overline{R},\overline{A},\overline{B})$, i.e. we know that
\begin{align}
\label{eq:closeness}
\begin{aligned}
R &\in (\overline{R} - \hat{\delta}, \overline{R} + \hat{\delta}) \\
A &\in (\overline{A} - \hat{\delta}, \overline{A} + \hat{\delta}) \\
B &\in (\overline{B} - \hat{\delta}, \overline{B} + \hat{\delta})
\end{aligned}
\end{align}
for some small $\hat{\delta} = \hat{\delta}(\zeta_2) > 0$.
We observe that $B - R\epd < 0$ provided $\abs{\epd} < \overline{\ep}$ and therefore we express the solutions to \eqref{eq:cont_left3}-\eqref{eq:mom_2_right3} as functions of the values $\rho_\pm^\nu, m_{\pm 2}^\nu$ and parameter $\epd$:
\begin{align}
\nu_- &= \frac{A}{R} + \frac{1}{R}\sqrt{(-B+R\epd)\frac{\rho_+^\nu - 2}{2 - \rho_-^\nu}} \label{eq:numep}\\
\nu_+ &= \frac{A}{R} - \frac{1}{R}\sqrt{(-B+R\epd)\frac{2 - \rho_-^\nu}{\rho_+^\nu - 2}} \label{eq:nupep}\\
\beta &= \frac{-m_{-2}^\nu(\rho_+^\nu -2)-m_{+2}^\nu(2-\rho_-^\nu)}{R} +\frac{1}{R}\sqrt{(-B+R\epd)(2 - \rho_-^\nu)(\rho_+^\nu - 2)} \label{eq:betaep}\\
\ep_{2L} &= \nu_-^2(\rho_-^\nu - 2) - (\rho_-^\nu)^2 - \frac{(m_{-2}^\nu)^2}{\rho_-^\nu} + \frac{124 + 8\sqrt{13}}{9}.\label{eq:ep2lep}
\end{align}
We also note that due to continuous dependence of various formulas on data the fact that the admissibility inequalities \eqref{eq:E_left3}-\eqref{eq:E_right3} and subsolution inequalities \eqref{eq:sub_trace3}-\eqref{eq:sub_det3} are satisfied as strict inequalities for values $(R_\pm, M_\pm)$ and $\epd = 0$ implies that these inequalities will be satisfied also for $(\rho_\pm^\nu, m_{\pm 2}^\nu)$ and $\epd > 0$ provided $\hat{\delta}$ and $\overline{\ep}$ are sufficiently small.

\subsection{Solution to \eqref{eq:SSME6}}\label{ss:ODE}

We have to construct the subsolution in such a way that the equation \eqref{eq:SSME6}, which can be written as
\begin{equation}
\frac{\d}{\d t} \beta(t) = -\frac{\epd(t)}{l(t)},\label{eq:ODE}
\end{equation}
is satisfied. In order to solve this equation we choose to emphasize the dependence of all functions on $\epd$, hence obtaining
\begin{equation}
\beta(t) = \overline{\beta}(t,\epd(t)) \qquad l(t) = \overline{l}(t,\epd(t)).
\end{equation}
To simplify notation, in the rest of this section we skip writing bars over $\beta$ and $l$ and treat them as functions of two variables, $t$ and $\epd$. In particular we have
\begin{equation}
\frac{\d}{\d t} \beta (t,\epd(t)) = \de_t \beta (t,\epd(t)) + \de_\ep \beta (t,\epd(t)) \frac{\d}{\d t} \epd(t)
\end{equation}
and the equation \eqref{eq:ODE} becomes
\begin{equation}\label{eq:ODE2}
\de_\ep \beta (t,\epd(t)) \frac{\d}{\d t} \epd(t) = -\frac{\epd(t)}{l(t,\epd(t))} - \de_t \beta (t,\epd(t)).
\end{equation}
We complement \eqref{eq:ODE2} with the natural initial condition
\begin{equation}\label{eq:ODE2IC}
\epd(0) = 0.
\end{equation}
We denote
\begin{align}
f(t,\epd(t)) &:= \de_\ep \beta (t,\epd(t)) \label{eq:fdef0}\\
g(t,\epd(t)) &:= \de_t \beta (t,\epd(t)). \label{eq:gdef0}
\end{align}

Note that the existence of solution to \eqref{eq:ODE2} cannot be solved by the Picard-Lindel\"of theorem because beside other things the function $g(t,\epd)$ develops a singularity at $t=0$. Below we provide an existence proof based on the Banach fixed point argument with a special emphasis put upon the possible singularities of $g(t,\epd)$.
For $T>0$ we consider the Banach space 
$$
X_T = \{h\in C[0,T], h(0) = 0\}
$$
with the norm $\|h\|_{X_T} = \sup_{t\in [0,T]} |f(t)|$. For $r>0$ we denote $\mathcal B_{X_T}(r) = \{h\in X_T,\ \|h\|_{X_T}\leq r\}$.

We consider an operator $F$ on $X_T$ defined as $F(\dd) = \epd$ where $\epd$ is a solution to
\begin{equation}\label{ode.Tdef}
f(t,\dd(t)) \frac{\d}{\d t} \epd(t) = -\frac{\epd(t)}{l(t,\dd(t))} - g(t,\dd(t)), \qquad \epd(0)=0,
\end{equation}
on a time interval $[0,T]$. 
The solution to \eqref{ode.Tdef} can be written as 
\begin{equation}\label{ode.int}
\epd(t) = -\int_0^t  g(\tau,\dd(\tau)) f^{-1}(\tau,\dd(\tau)) \exp\left\{-\int_\tau^t l^{-1}(s,\dd(s)) f^{-1}(s,\dd(s)) {\rm d}s\right\}  {\rm d}\tau.
\end{equation}
for $t\in [0,T]$ as far as the integral on the right hand side exists. As we are concerned with a local existence only, for simplicity we assume here $T\leq \frac12$.

\begin{lemma}\label{l:ODE1} Let there be a constant $c>0$ such that $|g(t,\dd(t))|< \frac c{t |\log t|}$ for all $\dd$ with $\|\dd\|_\infty<\frac{|B|}{2|R|}$. Then there exists $T^*>0$ such that $F:\mathcal B_{X_{T^*}}\left(\frac{|B|}{2|R|}\right) \mapsto \mathcal B_{X_{T^*}}\left(\frac{|B|}{2|R|}\right)$. 
\end{lemma}
\begin{proof}
It suffices to show that the integral on the right hand side of \eqref{ode.int} is finite and that it converges to $0$ as $t\to 0$. It is easy to observe that for $\dd \in \mathcal B_{X_{T}}\left(\frac{|B|}{2|R|}\right)$ there exist constants $c_1>0$ and $c_2>0$ such that $l(t,\dd(t))\in (c_1t, c_2t)$. Similarly expressing $f(t,\dd(t))$ from \eqref{eq:betaep} we have
\begin{equation}\label{eq:fdef}
f(t,\dd(t)) = \frac{(2-\rho_-^\nu)(\rho_+^\nu-2)}{2\sqrt{(-B+R\dd)(2 - \rho_-^\nu)(\rho_+^\nu - 2)}}
\end{equation}
and thus $f(t,\dd(t))\in (c_1,c_2)$.

Thus,  \eqref{ode.int} and the H\"older inequality yield
\begin{multline}\label{ode.sol.est}
|\epd(t)| \leq c\int_0^t \frac 1{\tau|\log \tau|} \frac 1{c_1} \exp \left\{-\int_\tau^t \frac 1{c_2^2 s}{\rm d}s\right\} {\rm d}\tau = \frac{c}{c_1}\int_0^t \frac 1{\tau|\log\tau|} \frac{\tau^{c_2^{-2}}}{t^{c_2^{-2}}} {\rm d}\tau\\
\leq \frac{c}{c_1 t^{c_2^{-2}}} \left(\int_0^t \frac 1{\tau (\log \tau)^2}\dtau\right)^{1/2} \left(\int_0^t \tau^{c_2^{-4} - 1}\dtau\right)^{1/2}\\
= \frac{c}{c_1 t^{c_2^{-2}}} \sqrt{\frac{1}{|\log t|}} \frac{t^{c_2^{-2}}}{\sqrt {2c_2^{-2}}} = \frac{c}{c_1 \sqrt{2c_2^{-2}}} \sqrt{\frac{1}{|\log t|}}.
\end{multline}
It immediately follows that $\epd \to 0$ as $t\to 0$. Furthermore, we may choose $T^*>0$ in order to have
$$
 \frac{c}{c_1 \sqrt{2c_2^{-2}}} \sqrt{\frac{1}{|\log t|}}\leq  \frac{c}{c_1 \sqrt{2c_2^{-2}}} \sqrt{\frac{1}{|\log T^*|}} \leq \frac{|B|}{2|R|}.
$$
\end{proof}

\begin{lemma}\label{l:ODE2}
Let the assumptions of Lemma \ref{l:ODE1} be satisfied and let there be a constant $C > 0$ such that
\begin{equation}\label{eq:deltagest}
|g(t,\dd^1(t)) - g(t,\dd^2(t))| \leq   C \|\dd^1-\dd^2 \|_{L^\infty(0,T^*)}\frac{1}{t|\log t|}
\end{equation}
for all $t < T^*$ and $\dd^1,\dd^2 \in \mathcal B_{X_{T^{*}}}\left(\frac{|B|}{2|R|}\right)$. Then there exists $T^{**}\in (0,T^*]$ such that $F:\mathcal B_{X_{T^{**}}}\left(\frac{|B|}{2|R|}\right) \mapsto \mathcal B_{X_{T^{**}}}\left(\frac{|B|}{2|R|}\right)$ is a contraction.
\end{lemma}

\begin{proof}
Take $\dd^1,\dd^2 \in \mathcal B_{X_{T^{*}}}\left(\frac{|B|}{2|R|}\right)$ and let $\epd^1 = F(\dd^1)$ and $\epd^2 = F(\dd^2)$. We also use a notation $h^i = h(t,\dd^i)$ for $h = f,g,l$ and $i = 1,2$. We have
\begin{equation}\label{ode.dif}
\frac{\d}{\d t}(\epd^1 - \epd^2) = -\frac{\epd^1 - \epd^2}{l^1f^1} - \epd^2\left(\frac{1}{l^1f^1} - \frac{1}{l^2f^2}\right) - \left(\frac{g^1}{f^1} - \frac{g^2}{f^2}\right).
\end{equation}
We use a notation $L^1 = (l^1 f^1)^{-1}$, $L^2 = (l^2 f^2)^{-1}$, $G^1 = g^1(f^1)^{-1}$ and $G^2= g^2(f^2)^{-1}$. From \eqref{ode.dif} we deduce
\begin{equation*}
(\epd^1 - \epd^2)(t) = \int_0^t \left(-\epd^2 (L^1 - L^2) - (G^1 - G^2)\right)(\tau) \exp\left(-\int_\tau^t L^1(s){\rm d}s\right){\rm d}\tau
\end{equation*}
We have
$$
|\epd^1(t)- \epd^2(t)| \leq \int_0^t\left(|\epd^2|(\tau)|L_1 - L_2|(\tau) + |G_1 - G_2|(\tau) \right)\frac{\tau^{c_2^{-2}}}{t^{c_2^{-2}}}{\rm d}\tau
$$

We will estimate the differences $|L^1 - L^2|$ and $|G^1 - G^2|$ in terms of $|\dd^1 - \dd^2|$. First, it is easy to observe from \eqref{eq:fdef}
\begin{equation}\label{eq:deltafest}
|f(t,\dd^1(t)) - f(t,\dd^2(t))| \leq C \|\dd^1 - \dd^2\|_{L^\infty(0,T^*)}.
\end{equation}
Similarly we have
\begin{equation}\label{eq:ldef000}
l(t,\dd(t)) = \int_0^t \frac{1}{|R|}\left(\sqrt{(-B+R\dd)\frac{\rho_+^\nu - 2}{2 - \rho_-^\nu}} + \sqrt{(-B+R\dd)\frac{2 - \rho_-^\nu}{\rho_+^\nu - 2}}\right)
\end{equation}
and therefore
\begin{equation}\label{eq:deltalest}
|l(t,\dd^1(t)) - l(t,\dd^2(t))| \leq C \|\dd^1 - \dd^2\|_{L^\infty(0,T^*)} t.
\end{equation}
Plugging together \eqref{eq:deltafest} and \eqref{eq:deltalest} we get
\begin{equation}\label{eq:deltaLest}
|L^1(\tau) - L^2(\tau)| \leq \frac{|f^1 - f^2|}{l^1f^1f^2} + \frac{|l^1-l^2|}{f^2l^1l^2} \leq C\|\dd^1 - \dd^2\|_{L^\infty(0,T^*)} \frac{1}{\tau}
\end{equation}
which together with \eqref{ode.sol.est} yield
\begin{equation}\label{eq:deltaLest2}
|\epd^2(\tau)||L^1(\tau) - L^2(\tau)| \leq C\|\dd^1 - \dd^2\|_{L^\infty(0,T^*)} \frac{1}{\tau \sqrt{|\log \tau|}}.
\end{equation}
Using the assumption \eqref{eq:deltagest} we get also
\begin{equation}\label{eq:deltaGest}
|G^1(\tau) - G^2(\tau)| \leq \frac{|g^1||f^1 - f^2|}{f^1f^2} + \frac{|g^1-g^2|}{f^2} \leq  C\|\dd^1 - \dd^2\|_{L^\infty(0,T^*)} \frac{1}{\tau|\log \tau|}
\end{equation}
All together we end up with
\begin{equation}\label{eq:dif.epsilonu}
|\epd^1(t)- \epd^2(t)| \leq C \|\delta_1 - \delta_2\|_{L^\infty(0,T^*)}\left(\int_0^t\frac{1}{\tau | \log \tau|}\frac{\tau^{c_2^{-2}}}{t^{c_2^{-2}}}\dtau + \int_0^t\frac{1}{\tau \sqrt{|\log \tau|}}\frac{\tau^{c_2^{-2}}}{t^{c_2^{-2}}}\dtau \right). 
\end{equation}
the first integral on the right hand side can be handled as in \eqref{ode.sol.est}. For the second integral we use the H\"older inequality to get
\begin{multline}
\int_0^t\frac{1}{\tau \sqrt{|\log \tau|}}\frac{\tau^{c_2^{-2}}}{t^{c_2^{-2}}}\dtau \leq \frac 1{t^{c_2^{-2}}} \left(\int_0^t \frac 1{\tau(\log \tau)^2}\dtau\right)^{1/4}\left(\int_0^t \tau^{4/3{c_2^{-2}} - 1}\dtau\right)^{3/4}\\
= \frac1{\sqrt[4]{|\log t|}}\frac{1}{\sqrt[4]{4/3c_2^{-2}}^3}. 
\end{multline}
Consequently, \eqref{eq:dif.epsilonu} yields
\begin{equation}
|\epd^1(t)- \epd^2(t)| \leq C \|\delta_1 - \delta_2\|_{L^\infty(0,T^*)} \left(\frac1{\sqrt{|\log t|}} + \frac 1{\sqrt[4]{|\log t|}}\right)
\end{equation}
Now it is sufficient to choose $T^{**}\in (0,T^*]$ such that $C\left(\frac1{\sqrt{|\log t|}} + \frac 1{\sqrt[4]{|\log t|}}\right)<1$.
\end{proof}

As a consequence of the previous two lemmas and the Banach fixed point theorem we get
\begin{coro}\label{c:coro}
Under the assumptions of Lemmas \ref{l:ODE1} and \ref{l:ODE2} there exists $T^{**}>0$ and a solution $\epd \in C(0,T^{**})$ to \eqref{eq:ODE2}-\eqref{eq:ODE2IC} such that 
\begin{equation}
|\epd(t)| \leq \frac{C}{\sqrt{|\log t|}}.
\end{equation}
for some positive constant $C$.
\end{coro}

\subsection{Construction of the initial data}

Now we are ready to present the final argument of the proof of Theorem \ref{t:main}. We first prescribe the functions $(\rho_\pm^\nu,m_{\pm 2}^\nu)(t)$ in a suitable way, such that \eqref{eq:closeness} is satisfied. With these functions given, we solve the Rankine-Hugoniot conditions \eqref{eq:cont_left3}-\eqref{eq:mom_2_right3} together with the differential equation \eqref{eq:SSME6} as it is described in Sections \ref{ss:Gen} and \ref{ss:ODE}, in particular we obtain the curves $\widetilde{\nu}_\pm(t)$ defining the generalized fan partition by~\eqref{eq:numep} and \eqref{eq:nupep}. Then we use the method of characteristics to map the given functions $(\rho_\pm^\nu,m_{\pm 2}^\nu)(t)$ to the initial data $(\rho^0_\pm,m_{\pm 2}^0)(x_2)$ and we prove that initial data constructed this way will be of the type presented in Lemma \ref{l:f0h}, whence they will be generated by a smooth compression wave.


We use the relation between the solution $\lambda_1$ to the Burgers equation \eqref{eq:Bur} and the solution $(\rho, m)$ to the Euler equations \eqref{eq:EEm1}-\eqref{eq:EEm2}. Recalling \eqref{eq:CharSpeed} and \eqref{eq:RiemInv} we get 
\begin{equation}\label{eq:rhom}
\rho = \frac{(w_1 - \lambda_1)^2}{18} \qquad m_2 = \frac{(w_1-\lambda_1)^2(2\lambda_1-w_1)}{54}.
\end{equation}
We know that $(\rho,m)$ is a solution to the Euler system if $w_1$ is constant and $\lambda_1$ is a continuous solution to the Burgers equation. We set $w_1 \equiv 4\sqrt{2}$ and we prescribe the functions $\lambda_{1\pm}^\nu(t) = \lambda_1(t,\widetilde{\nu}_\pm(t))$ using the function $f_0$ defined in \eqref{eq:f0def} as follows
\begin{equation}\label{eq:lambda1m}
\lambda_{1-}^\nu(t) = \left\{\begin{split}
&f_0(t) + \sqrt{2}-\frac{\zeta_2}{T} \qquad &\text{ for } 0 < t < \delta \\
&\text{smooth monotone} \qquad &\text{ for } \delta < t < \delta' \\
&\sqrt{2} \qquad &\text{ for } t > \delta'
\end{split} 
\right.
\end{equation}
and 
\begin{equation}\label{eq:lambda1p}
\lambda_{1+}^\nu(t) = \left\{\begin{split}
&-f_0(t) - 2\sqrt{2} + \frac{\zeta_2}{T} \qquad &\text{ for } 0 < t < \delta \\
&\text{smooth monotone} \qquad &\text{ for } \delta < t < \delta' \\
&-2\sqrt{2} \qquad &\text{ for } t > \delta'
\end{split} 
\right.
\end{equation}
for some positive small $\zeta_2$, $\delta$ and $\delta' > \delta$, such that $\lambda_{1\pm}^\nu$ are monotone with monotone first derivatives. Functions $(\rho_\pm^\nu,m_{\pm 2}^\nu)(t)$ are then defined through \eqref{eq:rhom} with $\lambda_1 = \lambda_{1 \pm}^\nu$.


This choice of functions $\rho_\pm^\nu$ and $m_{\pm 2}^\nu$ ensures that there exists a constant $c > 0$ such that
\begin{equation}
|\de_t \rho_-^\nu| + |\de_t \rho_+^\nu| + |\de_t m_{-2}^\nu| + |\de_t m_{+2}^\nu| \leq c|f_0'(t)| \leq \frac{c}{t|\log t|}.
\end{equation}
It is then not very difficult to conclude from \eqref{eq:betaep} and \eqref{eq:gdef0} that both 
\begin{equation}\label{eq:gestimateproved}
|g(t,\dd(t))| \leq \frac{C}{t|\log t|}    
\end{equation}
 and \eqref{eq:deltagest} are satisfied and thus, as stated in Corollary \ref{c:coro}, there exists $T^{**}$ and a solution $\epd \in C(0,T^{**})$ to \eqref{eq:ODE2}-\eqref{eq:ODE2IC}.

This means, that on $(0,T^{**})$ we constructed functions $\nu_\pm(t)$, $\beta(t)$, $\ep_{2L}(t)$ and $\epd(t)$ with $\ep_{2R}(t) = \ep_{2L}(t) - \epd(t)$ satisfying \eqref{eq:SSME6} together with the Rankine-Hugoniot conditions \eqref{eq:cont_left3}-\eqref{eq:mom_2_right3} with values of $\rho_1$ and $K$ given by \eqref{eq:rKchoice}. 

Moreover, using the continuity argument, we know that the admissibility inequalities \eqref{eq:E_left3}-\eqref{eq:E_right3} and subsolution inequalities \eqref{eq:sub_trace3}-\eqref{eq:sub_det3} are satisfied as strict inequalities for $\epd = 0$ and values of $(\rho_\pm^\nu,m_{\pm 2}^\nu)$ are close to $(R_\pm, M_{\pm 2})$ and therefore there exists time $\delta_0 > 0$ such that all these inequalities are satisfied as well with $\epd(t)$ constructed above.

Therefore we constructed an admissible generalized fan subsolution on time interval $(0,\delta_0)$ with the initial data given by mapping the values of $(\rho_\pm^\nu,m_{\pm 2}^\nu)$ from points $x_2 = \widetilde{\nu}_\pm(t)$ to $t = 0$ using the method of characteristics for the Burgers equation. It remains to show that initial data constructed in this way satisfy the assumptions of Lemma \ref{l:f0h} and therefore are generated by a smooth compression wave.

\begin{lemma}\label{l:last}
Let $\lambda_{1 \pm}^\nu(t)$ be defined as in \eqref{eq:lambda1m} and \eqref{eq:lambda1p}. Then there exists $\lambda_1^0(x_2)$ defined as in \eqref{eq:BurICGR}-\eqref{eq:Propertiesfpm} with $\lambda_- = \sqrt{2}$, $\lambda_+ = -2\sqrt{2}$ and with functions $h_\pm(x_2)$ satisfying the assumptions of Lemma \ref{l:f0h}, namely $h_\pm(0) = 0$, $h_\pm'(0) = K_\pm$ where $\pm K_\pm > 0$, and \eqref{eq:hass} holds, such that $\lambda_{1 \pm}^\nu(t) = \lambda_1(t,\widetilde{\nu}_\pm(t))$, where $\lambda_1(t,x)$ is a solution to the Burgers equation \eqref{eq:Bur}-\eqref{eq:BurIC}.
\end{lemma}

\begin{proof}
First note that since the values of $\lambda_{1\pm}^\nu$ are close to $\lambda_\pm$, the graphs of $x_2 = \widetilde{\nu}_\pm(t)$ lie in the regions where the characteristics do not intersect and thus the problem is really well defined, see Lemma \ref{l:41}.

Let us present the argument for $\lambda_{1-}^\nu(t)$ which takes the values of $\lambda_1^0(x_2)$ for $x_2 < 0$. More precisely we have
\begin{equation}\label{eq:relation}
    \lambda_1^0(x_2) = \lambda_{1-}^\nu(h_-(x_2))
\end{equation}
where the mapping $h_-$ assigns to $x_2 \leq 0$ the time instant $h_-(x_2) \geq 0$ such that 
\begin{equation}\label{eq:hdef1}
    \widetilde{\nu}_-(h_-(x_2)) = x_2 + h_-(x_2)\lambda_1^0(x_2),
\end{equation}
i.e. the time instant where the characteristic line starting from the point $x_2$ crosses the graph of $y = \widetilde{\nu}_-(t)$. Note that more useful for us is to use \eqref{eq:relation} in order to rewrite \eqref{eq:hdef1} as
\begin{equation}\label{eq:hdef2}
    \widetilde{\nu}_-(h_-(x_2)) = x_2 + h_-(x_2)\lambda_{1-}^\nu(h_-(x_2)).
\end{equation}
Obviously the function $h_- \in C^\infty((0,\infty))$ and it holds $h_-(0) = 0$. We use \eqref{eq:hdef2} to express the derivatives of $h_-$, this way we obtain
\begin{equation}\label{eq:hder}
    h_-'(x_2) = \frac{1}{\widetilde{\nu}_-'(h_-(x_2)) - \lambda_{1-}^\nu(h_-(x_2)) - h_-(x_2)(\lambda_{1-}^\nu)'(h_-(x_2))}
\end{equation}
and thus (recalling also that we have $\nu_- = \widetilde{\nu}_-'$)
\begin{equation}
    h_-'(0) = \lim_{x_2 \rightarrow 0} h_-'(x_2) = \frac{1}{\nu_-(0) - \sqrt{2} + \frac{\zeta_2}{T}}.
\end{equation}
Here $\nu_-(0)$ is close to the value $\frac{-\sqrt{8}-\sqrt{26}}{3}$ from \eqref{eq:subsolRiem} and therefore $h_-'(0) = K_-$ for some $K_- < 0$. Next we calculate the second derivative of $h_-$ and we obtain
\begin{equation}\label{eq:hder2}
    h_-''(x_2) = - \frac{\nu_-'(h_-(x_2)) - 2(\lambda_{1-}^\nu)'(h_-(x_2)) - h_-(x_2)(\lambda_{1-}^\nu)''(h_-(x_2)) }{(\nu_-(h_-(x_2)) - \lambda_{1-}^\nu(h_-(x_2)) - h_-(x_2)(\lambda_{1-}^\nu)'(h_-(x_2)))^3}.
\end{equation}
As we already observed, the denominator on the right hand side of \eqref{eq:hder2} is bounded and bounded away from zero as $x_2 \rightarrow 0$. Concerning the numerator we know that as $t \rightarrow 0$ we have
\begin{equation}
    |(\lambda_{1-}^\nu)' (t)| + |t (\lambda_{1-}^\nu)'' (t)| \leq C( |f_0'(t)| + 
    |t f_0''(t)| ) \leq \frac{C}{t}P_2(\log t, \log|\log t|)
\end{equation}
with some rational function $P_2$. Therefore it remains to prove that $\nu_-'(t)$ has the same behavior close to zero. We recall here the expression $\eqref{eq:numep}$ which relates $\nu_-$ to $\epd$ and to $\rho_\pm^\nu, m_{\pm 2}^\nu$ (and thus to $\lambda_{1 \pm}^\nu$ and ultimately to $f_0(t)$)
\begin{equation}
    \nu_- = \frac{A}{R} + \frac{1}{R}\sqrt{(-B+R\epd)\frac{\rho_+^\nu - 2}{2 - \rho_-^\nu}}.
\end{equation}
Here $R, A, B$ are defined in \eqref{eq:Rdef}, \eqref{eq:Adef} and \eqref{eq:Bdef} respectively. It is not difficult to observe that writing $(\rho_\pm^\nu,R,A,B) = (\rho_\pm^\nu,R,A,B)(f_0(t))$, the quantities $\rho_\pm^\nu, R, A, B$ are polynomials in the variable $f_0$, see \eqref{eq:rhom}, \eqref{eq:lambda1m} and \eqref{eq:lambda1p}. In particular the derivatives of these quantities with respect to $f_0$ are also polynomials and evaluated in $f_0(t)$ are bounded as $t \rightarrow 0$. In order to estimate $\epd'(t)$ we use the equation \eqref{eq:ODE2} and the estimates we have already obtained in the proof of Lemma \ref{l:ODE1}, in Corollary \ref{c:coro} and in \eqref{eq:gestimateproved}
\begin{align}
\left|\frac{1}{l(t,\epd(t))}\right| &\leq \frac{C}{t}    \\
\left|\frac{1}{f(t,\epd(t))}\right| &\leq C \\
\left|g(t,\epd(t))\right| & \leq \frac{C}{t|\log t|} \\
\left|\epd(t)\right| &\leq \frac{C}{\sqrt{\log t}}.
\end{align}
Therefore we can conclude that
\begin{equation}
    |\nu_-'(t)| \leq C(|f_0'(t)| + |\epd'(t)|) \leq C\frac{1}{t}\overline{P}_2(\sqrt{|\log t|}, \log|\log t|)
\end{equation}
with some rational function $\overline{P}_2$.

Returning to \eqref{eq:hder2} we have shown that \eqref{eq:hass} holds for $n = 2$. It is not difficult to observe that higher derivatives of \eqref{eq:hder2} will produce accordingly higher derivatives of $f_0$ and that \eqref{eq:hass} holds indeed also for any $n > 2$. Lemma \ref{l:last} is proved.
\end{proof}


\section*{Acknowledgment}

O. Kreml, V. M\'acha and S. Schwarzacher were supported by the GA\v CR (Czech Science Foundation) project GJ17-01694Y in the general framework of RVO: 67985840.
E. Chiodaroli was supported by the Italian National Grant FFABR 2017.

\end{document}